\documentclass[12pt,a4paper]{amsart}
\usepackage{amsmath,amsthm,amsfonts,amscd,amssymb,eucal,latexsym,mathrsfs}
\usepackage[all]{xy}
\usepackage{latexsym, amssymb, amscd, mathrsfs, graphics, graphicx, array}

\theoremstyle{plain}
\newtheorem{theorem}{Theorem}[section]
\newtheorem{corollary}[theorem]{Corollary}
\newtheorem{lemma}[theorem]{Lemma}
\newtheorem{proposition}[theorem]{Proposition}
\newtheorem{fact}[theorem]{Fact}

\newtheorem{claim}[theorem]{Claim}
\theoremstyle{definition}
\newtheorem{remark}[theorem]{Remark}
\newtheorem{remarks}[theorem]{Remarks}

\newtheorem{example}[theorem]{Example}

\newtheorem{question}[theorem]{Question}

\newtheorem{definition}[theorem]{Definition}

 \DeclareMathOperator{\Aut}{Aut}

\def\max{\mathop{\mathrm{max}}\nolimits}

\newcommand{\RI}{{\mathbb R}}\newcommand{\QI}{{\mathbb Q}}
\newcommand{\ZI}{{\mathbb Z}}

\newcommand{\cF}{{\mathcal F}}

  \newcommand{\F}{\mathbf F} 
 \newcommand{\C}{C}
\newcommand{\G}{\Gamma}

\newcommand{\GL}{\mathrm {GL}}\newcommand{\PGL}{\mathrm {PGL}}

\newcommand{\FI}{{\bf{F}}}

\newcommand{\e}{\varepsilon}

\newcommand{\val}{\mathrm{val}}
\newcommand{\rk}{\mathrm{rk}}

\newcommand{\ssi}{\Longleftrightarrow}
\newcommand{\inj}{\hookrightarrow}
\newcommand{\surj}{\twoheadrightarrow}

\newcommand{\del}{\partial}

\usepackage{tikz}

\newcommand{\sq}{${7\over 4}$}

\date{\today}

\title{Removing chambers in Bruhat--Tits buildings} 

\author{Sylvain Barr\'e}
\address{\hskip-\parindent
Sylvain Barr\'e, Universit\'e de Bretagne Sud, Universit\'e Europ\'eenne de Bretagne, France}
\email{sylvain.barre@univ.ubs.fr}
\author{Mika\"el Pichot}
\address{\hskip-\parindent
Mika\"el Pichot, Dept. of Mathematics \& Statistics, McGill University, Montreal, Quebec, Canada H3A 2K6}
\email{pichot@math.mcgill.ca}

\begin{document}


\begin{abstract}
We introduce and study a family of countable groups constructed from Euclidean buildings by ``removing" suitably chosen subsets of chambers. 
\end{abstract}

\maketitle

In this paper $X$ denotes a CAT(0) space and $\G$ is a countable group acting properly isometrically on $X$ with $X/\G$ compact. 

The idea is the following. We start with a Euclidean building $X$ of rank 2, which we see as a space of ``maximal rank",  and remove chambers from $X$ equivariantly with respect to the acting group $\G$. Taking a  universal cover of the resulting space,  this leads a new family of groups, which typically are extensions of the given group $\G$, and a new class of CAT(0) spaces, which we call \emph{building with chambers missing}. 

 The ``rank" of these new spaces is \emph{a priori} lower than that of the initial building. The basic reason for that, of course, is that all apartments containing the deleted chambers have disappeared.   In some cases, the rank will decrease in a controlled way. One might expect, for example, that the least the proportion of removed chambers is, the closest the rank of these spaces is from the initial buildings.   
These  groups and spaces are examples of objects ``of intermediate rank" in the sense of \cite{rd}. Buildings $X$ are viewed as spaces of maximal rank  among their rank interpolating siblings (e.g.\ the triangle spaces in the $\tilde A_2$ case).

Removing chambers in buildings leads to a rich supply of  groups and spaces of  intermediate rank  on which the following  alternative can be tested:  

\begin{center}
$\G$ is hyperbolic   $\leftrightarrow\G$ contains a copy of $\ZI^2$ 
\end{center}
see Section 6.B$_3$ in \cite{gromov1}.  This problem is one of our original motivation, for this paper, and for ``rank interpolation" in general.

Let us now describe our main results.

 

\subsection*{Buildings with chambers missing} 
A building with chambers missing consists of a simplicial complex $X$ endowed with a free action of a countable group $\G$ with compact quotient, which satisfies certain axioms for chambers removal described in Section \ref{defs}, Definition \ref{bmiss}. 
As for usual Tits buildings, they are (in the non degenerate case) organized into types, according to the Coxeter diagram associated with them (which is inherited from the building they come from). Accordingly, we speak of  spherical or Euclidean building with chambers missing when the diagram is finite or Euclidean.  

 Euclidean buildings with chambers missing can be endowed with a natural piecewise linear metric, which is CAT(0) precisely when $X$ is of dimension 2, see Section \ref{defs}, if one chamber at least is missing (if no chamber is missing, i.e.\ if $X$ is a Euclidean building, then the metric is always CAT(0) by well-known results of Bruhat and Tits).  Henceforth we assume that $X$ is two dimensional, endowed with its natural CAT(0) metric, and that $\G$ acts on $X$  properly isometrically with compact quotient. 
 
 \subsection*{Rank interpolation} Our goal is to explore the world of groups and spaces interpolating between the higher rank Euclidean buildings and the hyperbolic spaces using this operation of surgery of chambers on buildings. Our first result is that this idea collapses in some, presumably rare, cases. 
 Namely, we show that the rank can drop abruptly from 2 (the rank of the original building) to 1 (namely, hyperbolicity) by removing (equivariantly)  \emph{a single chamber} from a Euclidean building. 

 \begin{theorem}\label{Thm - Hyperbolic} There exists a Euclidean building $X$ of rank 2 and a group $\G$ acting properly isometrically on $X$ with compact quotient, such that the universal cover of the space obtained by removing the orbit of a {\bf single} chamber in $X$ is hyperbolic.  
 \end{theorem} 
 
 This situation is atypical. The idea is that removing (even equivariantly) only one chamber from a Bruhat--Tits building should produce spaces whose rank is ``close" to that of the original building, and in particular spaces acted upon by groups that are not hyperbolic. 
 
 \begin{question}
 Are there only finitely many examples of pairs $(X,\G)$ as in Theorem \ref{Thm - Hyperbolic}? What is the list of all possible such pairs $(X,\G)$?
 \end{question}
 
 We will give several constructions of  $(X,\G)$ with a single chamber removed along the text, but our list is finite.

\subsection*{Local rank}  The intermediate rank phenomenon can be detected locally,  asymptotically, or in between, i.e.\ at the mesoscopic level, as we have discussed in \cite{rd}. These three aspects are considered again in the present paper, since they all appear naturally in the process of removing chambers. 

We would like to quantify ``how intermediate" the rank of our buildings with chambers missing is. In Section \ref{CATrank}, we introduce a notion the {\bf local rank} for a CAT(0) space $X$ of dimension 2. The local rank is a rational number 
\[
\rk\colon G \mapsto \rk(G)\in [1,2]\cap \QI
\]
attached to any metric graph $G$ of girth $\geq 2\pi$. The extremal values are 1 and 2 which correspond respectively to $G$ having girth $>2\pi$ or $G$ being a (spherical) building. In {\bf all} other cases,  the fractional value is intermediate strictly in between 1 and 2. The local rank has a natural definition as a \emph{linear interpolation} between the two extremal cases (see Definition \ref{Def - local rank}).  If $G$ is the link of a vertex $x$ of $X$, the value $\rk(G)$ is a way to measure the \emph{proportion of flats} in the tangent cone of $X$ at $x$.
This notion encompasses the qualitative definition of \cite{rd} and produces precise numerical outputs. 
 The Moebius--Kantor graph that served our rank interpolation purposes in \cite{rd}, and that we introduced as the ``link of rank \sq" in \cite{rd}, turns out to have rank \sq\ also in the sense of Definition \ref{Def - local rank}, as a direct computation will show.

 \subsection*{The local-to-global problem}
 In general, it is unclear how a local assumption on the rank will develop into a large scale rank property of the space $X$.
  A well-known result of J. Tits shows that the situation can be understood if the local rank of $X$ is 2 (in the sense of the rank functional $\rk$ described above): if all the links of $X$ are spherical buildings, then $X$ is a Euclidean building.  (Compare the introduction of Section \ref{locglob}.) A similar statement is not true anymore for buildings with missing chambers:
 
 \begin{proposition}\label{No local to global - intro} There exists a simplicial complex, endowed with a free and  vertex transitive action of a countable group, which is {\bf not} a building with chambers missing, but all of whose links are buildings with chambers missing. 
\end{proposition}

 Nevertheless, whether or not $(X,\G)$ is a building with missing chambers \emph{can be detected locally}.  This is a consequence of the above mentioned result of Tits. In Theorem \ref{th:loccrit}, a practical version of this result is formulated by associating to $(X,\G)$ an ``invariant graph" which detects if there exists an extension \emph{and}, if there is one, provides an upper bound on the number of extensions when they exist. This allows to prove, for instance, the following result, which will be useful for our study our random groups in \cite{random}. 
  
\begin{theorem}\label{th:loccrit-intr}
Let $(X,\G)$ be a Euclidean building with chambers missing of type $\tilde A_2$ and assume that the distance between missing chambers in $X$ is at least 2. Then there is a {\bf  unique} extension $(X,\G)\leadsto (X',\G')$ of $(X,\G)$ into a Euclidean building $(X',\G')$.
\end{theorem}

This result shows that some form of rigidity remains in buildings with chambers missing (depending how chambers are removed). By Theorem \ref{th:loccrit-intr}, the space $(X,\G)$ ``remembers" entirely the building where it is coming from. Note that it applies in particular when $(X,\G)$ has {\bf only one} chamber missing. As we have already observed, in Theorem \ref{Thm - Hyperbolic},  the space $X$  itself \emph{may be hyperbolic} in this case. Yet it will remember the (rank 2) building that it is coming from.

\subsection*{Hyperbolicity vs the existence of $\ZI^2$}  Groups acting on Bruhat--Tits (Euclidean) buildings  are either hyperbolic or they contain $\ZI^2$ (see \cite{mostow}). As we already mentioned, a similar alternative is not known for more general groups acting properly on a CAT(0) space with compact quotient. Concrete cases where it is open are provided by groups of intermediate rank, for example for the groups of rank \sq\ of \cite{rd}. The idea of removing chambers from Bruhat--Tits buildings is directly related to this problem: if one removes ``only few" chambers from a building of higher rank, then the existence of flats in $X$ and/or $\ZI^2$ in $\G$ should to persist ``in general". If a ``significant quantity" of chambers is removed, then the situation is less clear.   In a followup paper \cite{random}, we develop these ideas further and introduce new family of \emph{random} groups that fits the present rank interpolation framework.

\subsection*{The spherical case}\label{intr:pyr}
Several examples of buildings with chambers missing constructed below are drawn from a class of groups acting on complexes whose links are spherical buildings with 3 chambers missing coming from the incidence graph $H$ of the Fano plane (i.e.\ the projective plane on the field $\F_2$ on two elements). Recall that $H$ is the unique spherical building of type $A_2$ and order 2, cf.\ e.g.\ \cite{Ronan}.

We show  that  (Proposition \ref{fact6prime}): 
\begin{proposition}
There exist exactly 6 isomorphism classes $G_1, G_2,\ldots, G_6$ of spherical buildings of type $A_2$ and order 2 with 3 chambers missing.
\end{proposition}

For each of them,  we study their basic properties and compute their local rank. Furthermore, these graphs  appear as links of CAT(0) complexes. This family is rich, even in the most transitive case:

\begin{theorem}
For each $i=1\ldots 6$, $i\neq 5$, there exists a CAT(0) simplicial complex $X$ with a free isometric action of $\G$ which is {\bf transitive on the vertices} of $X$, such that the vertex link of $X$ is isometric $G_i$. 
\end{theorem}

The proof of this result goes as follows: by inspecting the links $G_1,\ldots, G_6$, one can infer a finite set of possibilities for the \emph{shapes} (which are unions of equilateral triangles) of the complexes $X$ we are looking for. These 
shapes can be assembled by studying what are the allowed combinatorics for gluing the boundaries together. 

We will see that some of the resulting simplicial complexes have {\bf mesoscopic rank} in the sense of \cite{rd}. This means that they contain (exponentially many in the radius) concentric balls which are flat but cannot be extended into a flat.

There is an exceptional graph, $G_5$, 
 which can be represented as follows (Fig.\ \ref{IP}) and that we call the inverse pyramid. 
It has an obvious dihedral symmetry of order 6, which, together with the symmetry of order 2 fixing the top vertex and one of its neighbour, generates the full automorphism group. 

\begin{figure}[htbp]
\begin{tikzpicture}[line join=round]
\draw(0,-1.6)--(.378,.363);
\draw(0,-1.6)--(-.891,-.035);
\draw(0,-1.6)--(.513,-.328);
\draw(.378,.363)--(-.513,.328);
\draw(.891,.035)--(.378,.363);
\draw(-.513,.328)--(-.891,-.035);
\draw(0,1.6)--(-.513,.328);
\draw(.513,-.328)--(.891,.035);
\draw(0,1.6)--(.891,.035);
\draw(-.891,-.035)--(-.378,-.363);
\draw(-.378,-.363)--(.513,-.328);
\draw(0,1.6)--(-.378,-.363);
\shade[ball color=black] (0,-1.6) circle (0.3ex);
\shade[ball color=black] (.378,.363) circle (0.3ex);
\shade[ball color=black] (-.891,-.035) circle (0.3ex);
\shade[ball color=black] (.513,-.328) circle (0.3ex);
\shade[ball color=black] (-.513,.328) circle (0.3ex);
\shade[ball color=black] (0,1.6) circle (0.3ex);
\shade[ball color=black] (-.378,-.363) circle (0.3ex);
\shade[ball color=black] (.891,.035) circle (0.3ex);
\end{tikzpicture}
\caption{The graph $G_5$} \label{IP}
\end{figure}
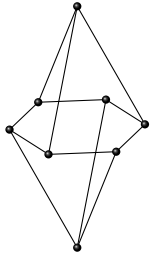

\noindent   For this graph the construction of a $V$ as above whose universal cover $X=\tilde V$ has all links isometric to $G_5$, as it turns out, is not possible:
 
\begin{proposition}
There is no simplicial complex $X$ whose link at every vertex is $G_5$.
\end{proposition}

 We refer to  Sect.\ \ref{spher}, Fig.\ \ref{Fano-3}, for a representation of the above six buildings with 3 chambers missing. The riddle concerning the existence of polyhedra with links isomorphic to  $G_5$ is discussed in Section \ref{Sectinvert}. Associated with the inverse pyramid are weak forms of  buildings with chambers missing (see e.g. d' in Theorem \ref{th101}) which are interesting in their own right.

\subsection*{The Haagerup property}
 
 An important question concerns the status of the Haagerup property and Kazhdan's property T for groups of intermediate rank. Since the groups to start with are lattices in algebraic group of higher rank (over local fields) they have the Kazhdan property T, and it is interesting to see if the chambers  removal surgery can lead to groups with property T and groups with the ``opposite" property, namely the Haagerup  property (also known as Gromov's a-T-menability). Observe  that if we remove \emph{all} the chambers, then the resulting group is a free group, which does have the Haagerup property.

We give in Section \ref{Sectinvert} an explicit construction  of  a ``weak" building with chambers missing $(X,\G)$, where $\G$ has the Haagerup property  (see Theorem \ref{th-haagerup}):

\begin{theorem} The group $\G$ with 4 generators and the 3 relations defined by
\[
\G:=\langle x,y,z,t\mid xyz,\, xzy,\, xt^{-1}yt^{-1}zt^{-1}\rangle
\]
has the Haagerup property.
\end{theorem}

The complex $X$ described on Figure \ref{fig-groupG} is  a CAT(0) complex with the isolated flat property (see  \cite{HK}) but it is not hyperbolic and  $\G$  contains $\ZI^2$. This properties are established in Section \ref{Sectinvert}. For the definition of  weak buildings with chambers missing, see \ref{weakbmiss}. To prove the Haagerup property, we construct explicit walls in $X$ and apply the classical criterion \cite{HP} of Haglund--Paulin. 
 
 \begin{figure}[h!]
\begin{tikzpicture}
\begin{scope}[xshift=-2.8cm,yshift=-.2cm]
\coordinate [label=left: {$\G=\langle 1,2,3,4\mid$}] (A1) at (0,0);
\end{scope}

\begin{scope}[xshift=-2.8cm,yshift=-.2cm]
\coordinate [label=left: {$\rangle$}] (A2) at (5.9,0);
\end{scope}

\begin{scope}[xshift=1.5cm,yshift=-.2cm,scale=.9]
\path (0:1cm) coordinate (P0);
\path (1*60:1cm) coordinate (P1);
\path (2*60:1cm) coordinate (P2);
\path (3*60:1cm) coordinate (P3);
\path (4*60:1cm) coordinate (P4);
\path (5*60:1cm) coordinate (P5);

\draw [-latex] (P0) -- node[anchor=west]{\tiny{1}}  (P1);
\draw [-latex] (P2) -- node[anchor=south]{\tiny{4}}  (P1);
\draw [-latex] (P2) -- node[anchor=east]{\tiny{2}}  (P3);
\draw [-latex] (P4) -- node[anchor=east]{\tiny{4}}  (P3);
\draw [-latex] (P4) -- node[anchor=north]{\tiny{3}}  (P5);
\draw [-latex] (P0) -- node[anchor=west]{\tiny{4}}  (P5);
\foreach \i in {0,...,5} \shade[ball color=black] (P\i) circle (0.3ex);
\end{scope}

\begin{scope}[xshift=-2.5cm,yshift=-.5cm,scale=.9]
\draw[-latex] (1,0) -- node[anchor=west]{\tiny{2}} (0.5,0.866);
\draw[-latex] (0.5,0.866) -- node[anchor=east]{\tiny{3}} (0,0);
\draw[-latex] (0,0) -- node[anchor=north]{\tiny{1}} (1,0);
\shade[ball color=black] (0,0) circle (0.3ex);
\shade[ball color=black] (1,0) circle (0.3ex);
\shade[ball color=black] (0.5,0.866) circle (0.3ex);
\end{scope}

\begin{scope}[xshift=-.9cm,yshift=-.5cm,scale=.9]
\draw[-latex] (1,0) -- node[anchor=west]{\tiny{3}} (0.5,0.866);
\draw[-latex] (0.5,0.866) -- node[anchor=east]{\tiny{2}} (0,0);
\draw[-latex] (0,0) -- node[anchor=north]{\tiny{1}} (1,0);
\shade[ball color=black] (0,0) circle (0.3ex);
\shade[ball color=black] (1,0) circle (0.3ex);
\shade[ball color=black] (0.5,0.866) circle (0.3ex);
\end{scope}
\end{tikzpicture}
\caption{The group $\G$ as a CAT(0) group}\label{fig-groupG}
\centerline{(the relations are read on the boundary of}
\centerline{the faces, respecting the orientation) }
\end{figure}
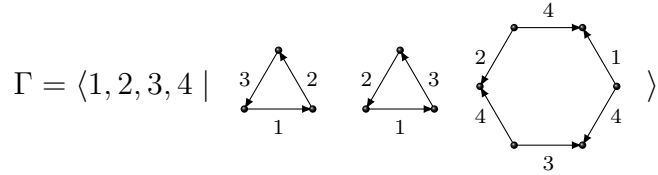

By way of contrast, it is found in \cite{random} that for ``most" of the buildings with chambers missing $(X,\G)$, the group $\G$ has the property T of Kazhdan.

\section{The definition of buildings with missing  chambers}\label{defs}

Let $X$ be a simplicial complex and let $\G$ be a group of automorphisms of $X$ acting freely with finite fundamental domain. 
We want to view $X$ as coming from a building $X'$ where some chambers have been removed equivariantly. This is done as follows:

\begin{definition}\label{bmiss}  We call $(X,\G)$  a \emph{building  with chambers missing} if there exist a Tits building $X'$, a group $\G'$ of automorphisms of $X'$ acting freely with finite fundamental domain,  a simplicial map $\lambda : X\to X'$, and chambers $C_1,\ldots, C_n$ in $X'$ such that: 
\begin{itemize}
\item[a)] $\G'C_i\neq \G'C_j$ and $\lambda(X)\cap \stackrel{\circ}{C_i}=\emptyset$ for $i\neq j\geq 1$
\item[b)] $\lambda(X)\cup  (\G'\stackrel{\circ}{C_1})\cup\ldots \cup (\G'\stackrel{\circ}{C_n})= X'$
\item[c)] $\lambda$ preserves orbits, in the sense that $x\sim_\G y\ssi \lambda(x)\sim_{\G'}  \lambda(y)$ for all $x,y\in X$.
\end{itemize}
\end{definition}

In the case that $X$ has non trivial boundary or is of non homogeneous dimension, we always add the mention ``building with chambers missing and with boundary". Otherwise $X$ is implicitly assumed to have no boundary and be of homogeneous dimension (equal to the dimension of $X'$).

\begin{remark}
 Depending on the context it can be interesting to generalize the scope of this definition. For instance one can allow proper actions, in order to include orbihedra, or insist on weaker versions on some of the  conditions a$\sim$c. 
  We won't need these generalizations here, except for some weakening of b) to be described in Definition \ref{weakbmiss}.
\end{remark}

A map $\lambda$ satisfying the conditions of Definition \ref{bmiss} is called an \emph{extension of $(X,\G)$ into a Tits building}, denoted in symbols 
\[
\lambda : (X,\G) \leadsto (X',\G').
\] 
 We call \emph{missing chamber} (resp. \emph{number of chambers missing}) of the extension $\lambda$ any chamber of $X'$ which doesn't belong to $\lambda(X)$  (resp., the integer $n$). 

The number of chambers missing of $(X,\G)$ is the minimal number of chambers missing taken over all possible extensions (this number is set to be zero if $X$ itself is a building). We say that two extensions $\lambda_0 : (X_0,\G_0) \leadsto (X_0',\G_0')$ and $\lambda_1 : (X_1,\G_1) \leadsto (X_1',\G_1')$ are isomorphic if there are equivariant isomorphisms  $\theta : X_0\to X_1$ and $\theta' : X_0'\to X_1'$ such that the following diagram commutes:

\centerline{\xymatrix{X_0 \ar@{<->}[d]^{\theta}\ar@{->}[r]^{\lambda_0}&X_0'\ar@{<->}[d]^{\theta'}\\X_1\ar@{->}[r]^{\lambda_1} &X_1'}}

\begin{definition}
We say that a building with chambers missing $(X,\G)$ is \emph{Euclidean} if $X$ is simply connected and if there exists an extension $\lambda : (X,\G) \leadsto (X',\G')$ into a Tits building, so that the building $X'$ is Euclidean.
\end{definition}

 In the one dimensional Euclidean case, namely when $X'$ is a tree, the space $X$ is a covering of a tree and in particular, is a tree itself. Therefore, buildings with chambers missing $(X,\G)$ in that case are merely buildings, and the resulting class of $\G$ consists only of {\bf free groups}. (Allowing more general  types of actions in Definition \ref{bmiss} would lead here to amalgamated free products \cite{serre}.)

The main result of the present section shows that Euclidean buildings with missing chambers are CAT(0) spaces if and only if the dimension equals 2 (here and below we do not distinguish here between complexes and their topological realization). 
This rises the problem of finding generalizations of these constructions to lattices in algebraic groups of rank $>2$ (possibly in some larger category than that of countable groups or countable groups acting on CAT(0) complexes). In particular, it would be interesting to find models that randomize these groups (producing {\bf random extensions} of lattices in algebraic groups of rank $>2$),  parallel to  the rank 2 case treated in \cite{random}.

\begin{theorem}\label{dim2}
Let $(X,\G)$ be a Euclidean building with at least one  missing chamber, let $\lambda : (X,\G) \leadsto (X',\G')$ be an extension, and consider the natural piecewise linear metric $d$ on $X$ obtained by pulling back the CAT(0) metric on $X'$. The following are equivalent:
\begin{enumerate}
\item the complex $X$ is of dimension 2;
\item the complex $X$ is contractible;
\item the metric space $(X,d)$ a CAT(0) space. 
\end{enumerate}
Furthermore in that case, the group $\G$ is a group of isometries of $(X,d)$.
\end{theorem}

\begin{proof}
As already mentioned in the one dimensional case  $X$ is a tree and in particular a Euclidean building. Therefore if $(X,\G)$ has at least one  missing chamber, then $\dim X\geq 2$.

We let $X_\bullet'\subset X'$ be the image of $\lambda$, namely,
\[
X_\bullet'=X\backslash  \left(\bigcup_{i=1}^n \G'  \stackrel{\circ}{\C_i}\right)
\]
where $C_1,\ldots C_n$ are the missing chambers. Consider the complexes $V=X/\G$ and $V'=X'/\G'$.  By condition c) of Definition \ref{bmiss}, $\lambda$ factorizes to an injective quotient map $\underline \lambda : V\to  V'$.    Fix a vertex $*$ in $V$ with respect to which fundamental groups are computed, and let 
\[
p_\lambda : \G\to \G'
\]
be the morphism between the respective fundamental groups $\G=\pi_1(V,*)$ of $V$ and $\G'=\pi_1(V',\underline\lambda(*))$ of $V'$ induced by $\underline \lambda$.   Note that the map 
\[
\lambda : X\surj X_\bullet'
\] 
is a covering map. The space $X$, being  simply connected,  is  the universal cover of $X_\bullet'$.

Let us show that (2) implies (1). Assume that $\dim X=\dim X'\geq 3$. Then $X_\bullet'$ is also simply connected, since $X'$ is contractible and the fundamental group of a space depends only upon its 2-skeleton.  
Therefore we have an homeomorphism
\[
X\simeq  X_\bullet'.
\] 
Since $X_\bullet'$ is not contractible (the boundary of a missing chamber is not homotopic to 0), this proves that (2) implies (1).

Let us now assume (1), i.e.\ $\dim X=2$, and show (3). 
Since the missing chambers of the extension $\lambda : (X,\G)\leadsto (X',\G')$ are of dimension $2$, the morphism $p_\lambda : \G\to \G'$ induced by $\lambda$ is surjective. More precisely, the missing  chambers $C_1,\ldots,C_n$ projects to topological disks $D_1,\ldots,D_n$ of $V'$, and if we choose a loop in $V$ from the base point $*$ around each of these (missing) disks, we get a generating set for a subgroup $N$ of $\pi_1(V,*)$, whose normal closure $\langle N\rangle$ satisfies 
\[
\pi_1(V',\underline\lambda(*))=\pi_1(V,*)/\langle N\rangle
\] 
i.e.\ we have an exact sequence

\everyentry={\vphantom{\langle N\rangle}}
\centerline{\xymatrix{1\ar[r]&\langle N\rangle\ar[r]& \G\ar[r]_{p_\lambda}&\G'\ar[r]& 1}}

\noindent which corresponds to the diagram:

\everyentry={\vphantom{X_\bullet'}}
\centerline{\xymatrix{X=\tilde V\ar[dd]_\G\ar[rrd]^{\lambda} \ar[rd]_{\langle N\rangle}&&\\ & {X_\bullet'\mbox{ }}\ar@{->}[r]\ar[ld]^{\G'} & X'=\tilde V'\ar[ld]^{\G'} \\
 {V\mbox{ }}\ar@{->}[r]_{\underline\lambda}&V' }}

\noindent where the horizontal arrows are factorizations of  $\lambda$. 
The metric $d$ on $X$ is given as usual by endowing each face of $X$ with the piecewise linear metric associated to the extension $\lambda$. Since the quotient by $\langle N\rangle$ is a covering map, the links of $X$ are isometric to those of $X_\bullet'$. Here both links are endowed with the angular metric associated to the Euclidean metric on faces. Therefore links of $X$ have girth at least $2\pi$. It follows that $X$ is locally CAT(0) for its natural length structure, and by the Hadamard-Cartan theorem, that $X$ a CAT(0) space. Thus (1) implies (3). 

That (3) implies (2)  is obvious, and that $\G$ acts by isometry follows from the piecewise linear definition of the metric in $X$.
\end{proof}

Unless otherwise specified we assume Euclidean buildings with chambers missing to be of dimension 2.

\begin{remark}\label{rem:nondegen} \begin{enumerate}
\item In the proof above, the assumption that $X'$ is a building was only used  to pull back the CAT(0) metric, and therefore the result generalizes to more general extensions into CAT(0) spaces (namely, to the case where the building $X'$ in Definition \ref{bmiss} is replaced by a finite dimensional CAT(0) space without boundary with underlying simplicial structure; for example, by a space of rank \sq\ in the sense of \cite{rd}). In many cases this gives, using the models of random groups presented in \cite{random}, a \emph{randomization of the corresponding class of groups} (resulting in random extensions of these groups), similar to the randomization of groups acting on buildings described in \cite{random}.

\item It follows  immediately from \ref{dim2}  that there is a canonical \emph{quasi-isometry class} of CAT(0) metrics on a building with chambers missing, which is given their extensions into Euclidean buildings (and therefore that notions such as hyperbolicity have an intrinsic meaning for a Euclidean building with missing chamber).
\end{enumerate}  
\end{remark}

\begin{definition}
We say that an extension $\lambda : (X,\G) \leadsto (X',\G')$ of a building with chambers missing $(X,\G)$  is \emph{non-degenerate} if there exists an apartment of $X'$ which contains no missing chamber.  
\end{definition}

By Remark \ref{rem:nondegen}, either all the extension of $(X,\G)$ are non degenerate or all are degenerate. We now show that  the Coxeter types actually coincide in the former case. 

\begin{lemma}\label{lem12} Let $(X_0,\G_0)$ and $(X_1,\G_1)$ a Euclidean buildings with chambers missing, and assume that we have two non-degenerate extensions  $\lambda_0 : (X_0,\G_0) \leadsto (X_0',\G_0')$ and $\lambda_1 : (X_1,\G_1) \leadsto (X_1',\G_1')$. If $X_1$ and $X_2$ are simplicially isomorphic, then $X_1'$  and $X_2'$ have the same Coxeter diagram.  
\end{lemma}
\begin{proof} 
Let $\rho : X_1\to X_2$ be a simplicial isomorphism between $X_1$ and $X_2$. By assumption some apartment $A_i$ of $X_i'$ (where we endow the buildings $X_i'$ with their maximal family of apartments) is included in the image $(X_i')_\bullet$ of $\lambda_i$.  Since apartments are contractible, we can choose a pull-back $\tilde A_1$ of $A_1$ under the covering map $X_1\to (X_1')_\bullet$ which is simplicially isomorphic to $A_1$. The image of $\rho(\tilde A_1)$ under the projection $X_2\to (X_2')_\bullet$ provides us with a simplicial complex $A_1'$ of $(X_2')_\bullet$ and therefore of $X_2'$. The complex $A_1'$ is simplicially isomorphic to a Euclidean Coxeter complex of dimension 2. It follows that $A_1'$ is an apartment of the Euclidean buildings $X_2$, and hence that $X_1'$ and $X_2'$ have the same Coxeter diagram. 
\end{proof}

\begin{definition}
We say that a Euclidean building with chambers missing $X$  is \emph{non-degenerate} if it admits a non-degenerate extension into a Tits buildings.
To such a complex is attached a unique Coxeter diagram (by Lemma \ref{lem12}), called the Coxeter diagram of $X$.   
\end{definition}

The following corollary shows that non-degenerate Euclidean buildings with chambers missing have a \emph{canonical CAT(0) metric} (as opposed to a quasi-isometry class of metric), which is induced by any extension into a Euclidean building. All buildings with chambers missing are endowed with this metric.

\begin{corollary}
Let $(X,\G)$ be a non-degenerate Euclidean buildings with chambers missing of dimension 2. Then there is one and only one piecewise Euclidean metric on $X$ such that any extension $(X,\G) \leadsto (X',\G')$ into a Euclidean building induces an isometry on faces.
\end{corollary}

\begin{proof}
This is a direct application of Lemma \ref{lem12}, and the fact that a simplicial isomorphism of Coxeter complexes is an isometry for their corresponding CAT(0) metric.  Concretely, the faces of $X$ have the following shapes of the Euclidean plane: a square in the $\tilde A_1\times \tilde A_1$ case, an equilateral triangle in the $\tilde A_2$ case, a $(2,4,4)$ right-angled triangle in the $\tilde B_2$ case,
 a $(2,3,6)$ right-angled triangle in the $\tilde G_2$ case.
\end{proof}

\begin{remark}
We will see later (cf. Section \ref{sec2}) that in the degenerate case,  there exist buildings with chambers missing for which two natural length structures metrics coexist, namely, one where all faces are equilateral triangle, and the other one where all faces are $(2,4,4)$ triangles. In our examples, however, only one of this metric will be CAT(0) (although both are hyperbolic).   
\end{remark}

Finally we note that, in the Euclidean case, degeneracy is synonymous to hyperbolicity:

\begin{corollary}
A Euclidean building with chambers missing of dimension 2 is degenerate if and only if it is hyperbolic.  
\end{corollary}

\begin{proof}
This is straightforward by the no flat criterion \cite{BH} and the fact that apartments (of a maximal system of apartments) coincide with flat subspaces of dimension 2.
\end{proof}

\section{The spherical case}\label{spher}

We now turn to the spherical case.   
Our objective is the classification certain building with chambers missing of type $A_2$, namely buildings with \emph{three chambers missing} (Proposition \ref{fact6prime}). 
In the spherical case, our standing freeness assumption implies that the acting group $\G$ is trivial, so we abbreviate the notation $(X,\G)$ to $G$. 

\begin{definition}\label{bmissph} A  finite simplicial complex $G$ is said to be a \emph{spherical building with chambers missing} if there is an extension $G\leadsto G'$, in the sense of Definition \ref{bmiss}, where the Tits building $G'$ is spherical. Unless otherwise mentioned we assume that $G$ has no boundary.
 \end{definition}

\begin{remark}
Again, more general situations involving finite groups acting on finite complexes, or even infinite groups acting on general spherical Tits buildings (boundaries),  could be considered. What we study in this paper are vertex links obtained from chambers surgery in Euclidean buildings.
\end{remark}

If $(X,\G)$ is a Euclidean building with chambers missing, then the vertex links of $X$ are a spherical building with (possibly no) chambers missing (and possibly with boundary if $(X,\G)$ is so). Contrary to the case of buildings, the converse is not true; for a counterexample, see Section \ref{locglob}. In some circumstances however, one can give concrete conditions under which it holds (see Theorem \ref{th:loccrit}).

 As in the Euclidean case, if a building with missing chamber  $G$ is non-degenerate, then it has a Coxeter diagram, according to the following fact: If $G$ admits non-degenerate extensions into Tits buildings  $G\leadsto G_i'$, $i=0,1$, then $G_0'$  and $G_1'$ have the same Coxeter diagram.  The proof is similar to its Euclidean analog.

\begin{remark} 
\begin{enumerate}
\item  Degenerate extensions of buildings with chambers missing may have different Coxeter diagram. For instance, remove a set of chambers in a finite building of dimension 1, say of type $A_2$, and let $G$ be the corresponding graph. This is spherical building with chambers missing. Furthermore, being bipartite, this graph embeds into the complete bipartite graph. This embedding is degenerate as soon as $G$ is non-degenerate.      
\item Let $(X,\G)$ be a Euclidean building with chambers missing. Assume that $\G$ is transitive on vertices and let $L$ be the link of $X$. Then $L$ is a building with at least 3 chambers missing.
\end{enumerate}
\end{remark}

Recall that there is a unique building of type $A_2$ and order 2, namely, the incidence graph $H$ of the Fano plane. We recall that every building of type $A_2$ is associated to a projective plane (possibly exotic), whose order is called the \emph{order} of the building.  For a building with missing chamber $G$ of type $A_2$, we call order of $G$ the order of an extension into a building of type $A_2$ (this doesn't depend on the choice of the extension, compare Section \ref{locglob}). 
We now classify the buildings of type $A_2$ and order 2 with few chambers missing.

\begin{proposition}\label{fact6prime} Up to isomorphism there exist exactly:
\begin{enumerate}
\item one spherical building of type $A_2$, order 2, and a single chamber missing;
\item two spherical buildings of type $A_2$, order 2, and two chambers missing;
\item six spherical buildings of type $A_2$, order 2, and three chambers missing.
\end{enumerate}
\end{proposition}

The graphs $G_1,\ldots, G_6$ corresponding to the third assertion are shown on Fig.\ \ref{Fano-3} (in the usual occidental direction of reading). The graph $G_5$ is the inverse pyramid.

\begin{figure}[htbp]
\centerline{\includegraphics[width=3.2cm]{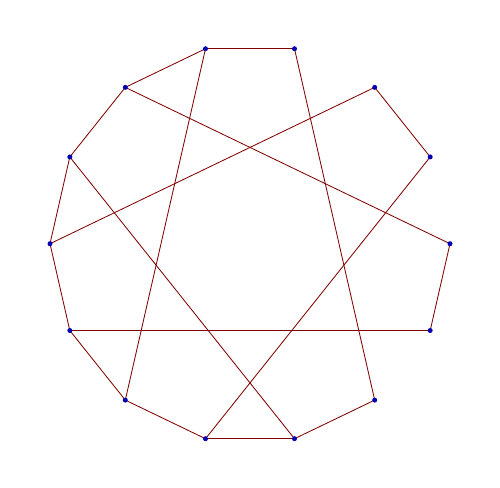}\includegraphics[width=3.2cm]{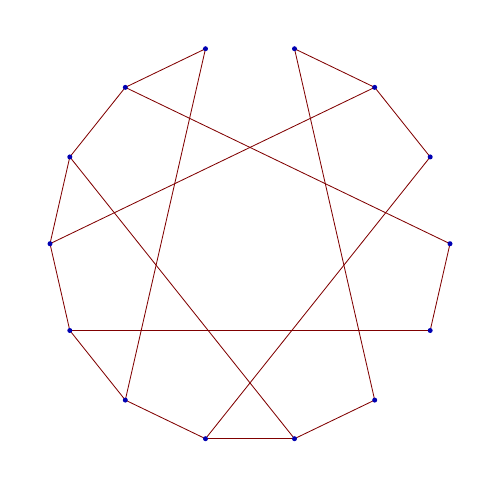}\includegraphics[width=3.2cm]{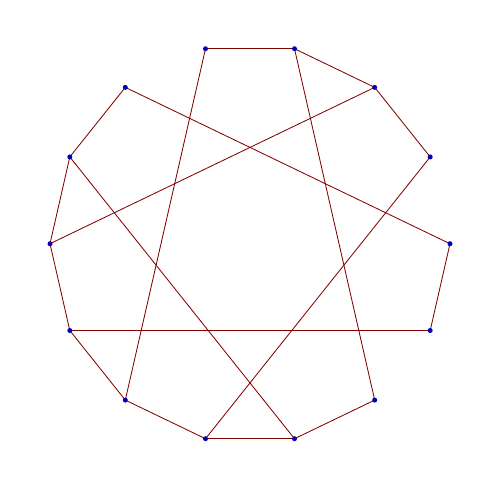}}\centerline{\includegraphics[width=3.2cm]{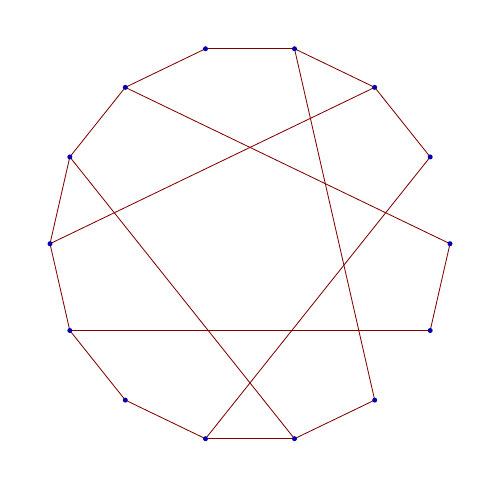}\includegraphics[width=3.2cm]{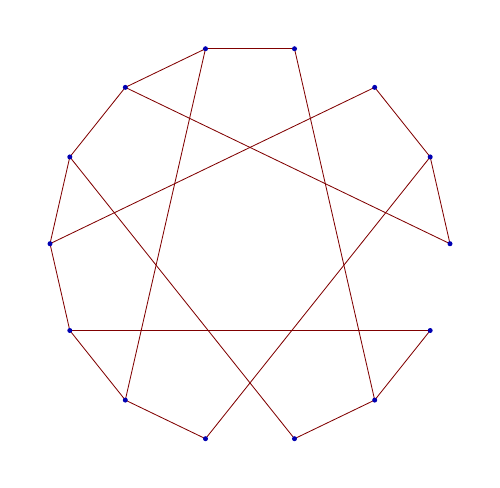}\includegraphics[width=3.2cm]{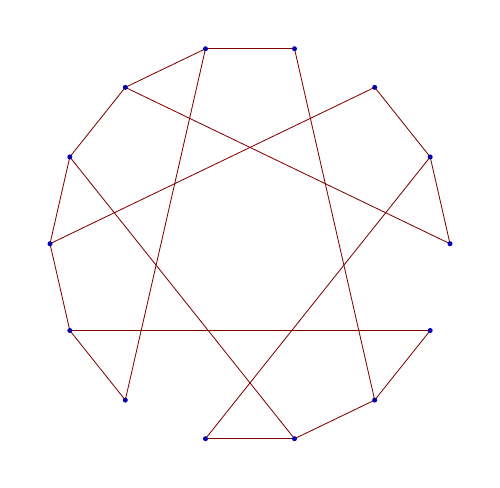}}
\caption{Spherical buildings of type $A_2$  with 3 chambers missing ($q=2$)} \label{Fano-3}
\end{figure}

\begin{proof}
The first assertion is obvious from the fact that $\Aut(H)=\PGL(2,7)$ is edge transitive. 

The second assertion follows from  $H$ being a  building. Indeed, let $e$ and $f$ be two edges of $H$. Then there exists an apartment $A$ of $H$ (i.e.\ a 6-cycle) containing both $e$ and $f$. Since by definition buildings with chambers missing have no terminal edge, $e$ and $f$ are disjoint in $A$. The two types of buildings with two chambers missing depend on whether $e$ and $f$ are opposite  or not in $A$, as is readily checked.

Let us prove the third assertion.  
We claim that, up to  $\Aut(H)$,  $H$ contains exactly 6 subsets of 3 pairwise disjoint edges. 
Let $E=\{e, f,g\}$ be a subset of 3 pairwise disjoint edges of $H$. We distinguish four cases.

Case $(1,1,1)$: there is an apartment $A$ containing $E$. Since $H$ is a building, $\Aut(H)$ is transitive on flags chambers $\subset$ apartments. Therefore there is exactly one solution in this case. We call $G_1$ the resulting building with chambers missing. 

Case $(1,1,2)$: exactly two edges of $E$, say $e$ and $f$, are at distance 2. In that case $g$ is at distance 1 from both $e$ and $f$. As in the second assertion above, we may assume that $e$ and $f$  are opposite in some apartment $A$.  Once $A$ is fixed, there are 4 possible choices for the edge $g$. However only 2 of these choices remains using easy symmetries of $H$, and therefore we get two buildings in that case, that we call $G_2$ and $G_3$. 

Case $(1,2,2)$: exactly two edges of $E$, say $e$ and $f$, are at distance 1. In that case $g$ is at distance 2 from both $e$ and $f$. One easily check that there are only two possible choices for $g$ once $e$ and $f$ are given. Since there is a symmetry interchanging the two situations, and since $\Aut(H)$ 
is transitive on paths of length 3, this case provides us with only one building with chambers missing. We call it $G_4$.

Case $(2,2,2)$: the pairwise distance between $e$, $f$ and $g$ is at least (and hence is exactly) equal to 2. As in $(1,1,2)$ there exists an apartment $A$ containing both $e$ and $f$ as opposite edges.  It is then easily seen we have 3 possible edges at distance 2 from both $e$ and $f$. Two of these 3 subcases gives isomorphic graphs (using a symmetry which flips $A$ stabilizing $e$ and $f$).  This gives us two solutions $G_5$ and $G_6$.

The graphs $G_1,\ldots,G_6$ are pairwise non isomorphic. To see this, we note that the family of lengths of maximal non-branching paths counted with multiplicity (the \emph{length spectrum} of $G_i$)  distinguishes them from each other. Indeed we have respectively for $G_1,G_2,G_3$:
\[
\{({\pi\over 3}, 9),(\pi,3)\}, \ \{ ({\pi\over 3}, 8),({2\pi\over 3},2),({\pi},2)\},\  \{ ({\pi\over 3}, 8),({2\pi\over 3},3),({4\pi\over 3},1)\}
\]
and for $G_4,G_5,G_6$:
\[
\{ ({\pi\over 3}, 7),({2\pi\over 3},4),(\pi,1)\},\  \{ ({\pi\over 3}, 6),({2\pi\over 3},6)\},\ \{ ({\pi\over 3}, 6),({2\pi\over 3},6)\}
\]
as length spectra. To distinguish between $G_5$ and $G_6$ we note that for example that $G_5$ has an apartment made of paths of length $\pi\over 3$ while $G_6$ does not. 
This concludes the proof of (3).
\end{proof}

\section{The case where one chamber only is missing} \label{sec2}

Take a (thick) Euclidean building, which has rank 2, remove only one chamber from this building, equivariantly, in the sense of Definition \ref{bmiss}, taking a universal cover. Then it may happen that the resulting space is hyperbolic. This section is centered around this fact.

It is convenient in concrete situations to work with polyhedral complexes (with some finite set of shapes of the Euclidean plane) rather than simplicial complexes. One gets back to the former by producing a (usually canonical) triangulation.

\begin{theorem}\label{th2} There exist two Euclidean buildings with exactly one chamber missing $(X_0,\G_0)$  and $(X_1,\G_1)$ of dimension 2  such that: 
\begin{enumerate}
\item both complexes $X_i$ are hyperbolic, $\G_i$ acts freely on $X_i$ transitively on vertices;
\item all links in $X_0$ and $X_1$ are isomorphic to a same graph $G_6$,  which is a spherical building with 3 chambers missing;
 \item there is a unique extension $(X_i,\G_i)\leadsto (X_i',\G_i')$ 
of $(X_i,\G_i)$ into a Euclidean building of type $\tilde A_2$, in the sense of Definition \ref{bmiss}, where the building $X_i'$ is endowed with a free action of a countable group $\G_i'$ which is transitive on vertices;
 \item the buildings  $X_0'$ and $X_1'$ are not isometric.
\end{enumerate} 
\end{theorem}

Let us consider the following two CW complexes $V_0$ and $V_1$, obtained by gluing together three lozenges according to their boundary, respecting orientations, as follows.\\
\begin{figure}[h]
\centering
\begin{tikzpicture}[y=.8cm]
\begin{scope}[xshift=-3cm]
\draw[-latex] (1,0) -- node[anchor=west]{\tiny{1}} (0.5,0.866);
\draw[-latex] (0.5,0.866) -- node[anchor=east]{\tiny{2}} (0,0);
\draw[-latex] (1,0) -- node[anchor=west]{\tiny{2}} (0.5,-0.866);
\draw[-latex] (0.5,-0.866) -- node[anchor=east]{\tiny{3}} (0,0);
\shade[ball color=black] (0,0) circle (0.3ex);
\shade[ball color=black] (1,0) circle (0.3ex);
\shade[ball color=black] (0.5,0.866) circle (0.3ex);
\shade[ball color=black] (0.5,-0.866) circle (0.3ex);
\draw[-latex] (2.6,0) -- node[anchor=west]{\tiny{1}} (2.1,0.866);
\draw[-latex] (2.1,0.866) -- node[anchor=east]{\tiny{4}} (1.6,0);
\draw[-latex] (2.6,0) -- node[anchor=west]{\tiny{3}} (2.1,-0.866);
\draw[-latex] (2.1,-0.866) -- node[anchor=east]{\tiny{2}} (1.6,0);
\node at (2.1,-1.6) {$V_6^0$};
\shade[ball color=black] (1.6,0) circle (0.3ex);
\shade[ball color=black] (2.6,0) circle (0.3ex);
\shade[ball color=black] (2.1,0.866) circle (0.3ex);
\shade[ball color=black] (2.1,-0.866) circle (0.3ex);
\draw[-latex] (4.2,0) -- node[anchor=west]{\tiny{1}} (3.7,0.866);
\draw[-latex] (3.7,0.866) -- node[anchor=east]{\tiny{3}} (3.2,0);
\draw[-latex] (4.2,0) -- node[anchor=west]{\tiny{4}} (3.7,-0.866);
\draw[-latex] (3.7,-0.866) -- node[anchor=east]{\tiny{4}} (3.2,0);
\shade[ball color=black] (3.2,0) circle (0.3ex);
\shade[ball color=black] (4.2,0) circle (0.3ex);
\shade[ball color=black] (3.7,0.866) circle (0.3ex);
\shade[ball color=black] (3.7,-0.866) circle (0.3ex);
\end{scope}

\begin{scope}[xshift=3cm]
\draw[-latex] (1,0) -- node[anchor=west]{\tiny{1}} (0.5,0.866);
\draw[-latex] (0.5,0.866) -- node[anchor=east]{\tiny{2}} (0,0);
\draw[-latex] (1,0) -- node[anchor=west]{\tiny{3}} (0.5,-0.866);
\draw[-latex] (0.5,-0.866) -- node[anchor=east]{\tiny{4}} (0,0);
\shade[ball color=black] (0,0) circle (0.3ex);
\shade[ball color=black] (1,0) circle (0.3ex);
\shade[ball color=black] (0.5,0.866) circle (0.3ex);
\shade[ball color=black] (0.5,-0.866) circle (0.3ex);
\draw[-latex] (2.6,0) -- node[anchor=west]{\tiny{1}} (2.1,0.866);
\draw[-latex] (2.1,0.866) -- node[anchor=east]{\tiny{3}} (1.6,0);
\draw[-latex] (2.6,0) -- node[anchor=west]{\tiny{4}} (2.1,-0.866);
\draw[-latex] (2.1,-0.866) -- node[anchor=east]{\tiny{2}} (1.6,0);
\shade[ball color=black] (1.6,0) circle (0.3ex);
\shade[ball color=black] (2.6,0) circle (0.3ex);
\shade[ball color=black] (2.1,0.866) circle (0.3ex);
\shade[ball color=black] (2.1,-0.866) circle (0.3ex);
\node at (2.1,-1.6) {$V_6^1$};
\draw[-latex] (4.2,0) -- node[anchor=west]{\tiny{1}} (3.7,0.866);
\draw[-latex] (3.7,0.866) -- node[anchor=east]{\tiny{4}} (3.2,0);
\draw[-latex] (4.2,0) -- node[anchor=west]{\tiny{2}} (3.7,-0.866);
\draw[-latex] (3.7,-0.866) -- node[anchor=east]{\tiny{3}} (3.2,0);
\shade[ball color=black] (3.2,0) circle (0.3ex);
\shade[ball color=black] (4.2,0) circle (0.3ex);
\shade[ball color=black] (3.7,0.866) circle (0.3ex);
\shade[ball color=black] (3.7,-0.866) circle (0.3ex);
\end{scope}
\end{tikzpicture}
\caption{}
\end{figure}

The fundamental groups of these complexes have the following presentations:
\[
\G_6^0=\langle u,v\mid v  u^2  v^{-1}  u  v^{-1}  u^{-1}  v  =uv^2u^{-1} \rangle
\]
and 
\[
\G_6^1=\langle u,v,w\mid u^2 w    = vwv,\,
    w   =v^{-1}uw^{-1}uv\rangle
\]
One readily sees that\\
\centerline{$H_1(V_6^0,\ZI)=\ZI\times\ZI/2\ZI$ and $H_1(V_6^1,\ZI)=\ZI\times\ZI/2\ZI\times\ZI/2\ZI$;}\\ 
in particular $\G_6^0$ and $\G_6^1$ are not isomorphic.

\begin{lemma}\label{l51}
The complexes $V_6^0$ and $V_6^1$ have a single vertex whose link is a building of type $A_2$ with 3 chambers missing (namely, $G_6$). 
\end{lemma}

 Since the length spectrum of $G_6$ is $\{{\pi\over 3}, {2\pi\over 3}\}$ with multiplicity 6, it is natural to look for complexes whose links are isomorphic to $G_6$ among polyhedra are built out of three lozenges (compare Section \ref{Sectinvert}).

\begin{proof}
It is obvious from the above description of $V_6^i$ ($i=0,1$) that the length spectrum of the reunion $L_i$ of all links in $V_6^i$ is that of $G_6$. Further, one easily sees by a direct inspection (which we call \emph{face chasing}) that $L_i$ is connected in both cases.  It follows that $V_6^i$  has a single vertex. Now $G_6$ is characterized among graph with length spectrum  $\{ ({\pi\over 3}, 6),({2\pi\over 3},6)\}$ by the following properties:
\begin{itemize}
\item paths of length $2\pi\over 3$ are organized in a disjoint union of a circle $C$ of length $2\pi$ and a tripod $T_\mathrm{u}$;
\item paths of length $\pi\over 3$ form a tripod $T_\mathrm{d}$ isometric to $T_\mathrm{u}$ attached to the circle $C$. 
\end{itemize}
These two properties can be checked on $L_i$, again, by face chasing.
\end{proof}

\begin{lemma}
The universal covers $X_6^0=\tilde V_6^0$  and $X_6^1=\tilde V_6^1$  are hyperbolic CAT(0) spaces. 
\end{lemma}

\begin{proof}
That $X_6^i$ is CAT(0) spaces follows from $G_6$ having no cycle of length smaller that $2\pi$ (since we removed edges from a spherical building). 
Let us show that $X_6^i$ is hyperbolic. 

We first observe a property of $G_6$, in the notations of the proof of \ref{l51}: every cycle of length $2\pi$ which contains the center of  $T_\mathrm{u}$ must contain the center of $T_\mathrm{d}$. This can be shown easily.

By definition, each face of $X_6^i$ contains an oriented edge with label 1. Furthermore, one easily check that this edge corresponds either to the center of  $T_\mathrm{u}$ or $T_\mathrm{d}$ in links of $X_6^i$. 

The set of edges labelled 1 form a family of parallel geodesic of $X_6^i$ each of which is bounded by 3 flat strips of width $\sqrt3\over 2$. The opposite boundary of each of these three flats strips belongs to either: 
\begin{itemize}
\item[$i=0$:] two families of parallel geodesics with respective periodic labels 3 and $24$;
\item[$i=1$:] the family of parallel geodesics with periodic labels $432$.
\end{itemize}
In both cases, it follows that every flat strip in $X_6^i$ has width at most $\sqrt 3$. By the no flat criterion, this shows that $X_6^i$ is hyperbolic.
\end{proof}

To see that $(X_6^i,\G_6^i)$ are building with exactly one missing chambers, one has to glue exactly one face to the the complex $V_i$ and recognize a Euclidean building (observe that this face must be glued on the diagonals of lozenges). 

In fact we have: 

\begin{lemma}
Both $(X_6^0,\Gamma_6^0)$ and $(X_6^1,\Gamma_6^1)$ are buildings with (exactly) one missing chamber. Furthermore, in each case there is a \emph{unique} extension $(X_6^0,\Gamma_6^0)\leadsto ({X_6^0}',{\Gamma_6^0}')$ into a Euclidean building of type $\tilde A_2$.
\end{lemma}

That the extension exist can be proved directly, by understanding the gluing procedure indicated above, but we will simply refer to a more general criterion proved in Section \ref{locglob} (see Theorem \ref{th:loccrit}). The same applies to unicity (compare Corollary \ref{cor:unicity}), for which a direct argument follows using the fact that $G_6$ admits has unique extension into building of type $A_2$ (we recall that $H$ denotes the incidence graph of the Fano plane):

\begin{lemma}[Unicity of extensions for $G_6$]\label{l54}
Let $\lambda_1,\lambda_2: G_6\leadsto H$  be two extensions of $G_6$ into $H$. There exists an automorphism $\theta\in \Aut(H)$ such that $\theta\circ \lambda_1=\lambda_2$.
\end{lemma}

\begin{proof}
Let $\theta_0: \lambda_1(G_6)\to \lambda_2(G_6)$ be an isomorphism between the images (note that $G_6$ has non trivial automorphisms). It is easily seen that (in the notations of \ref{l51})  $\theta_0$  restricted to automorphisms $\theta_0':  T_\mathrm{u}^1 \to T_\mathrm{u}^2$ and $\theta_0'':  C^1 \to C^2$ between tripods and circles (which we see as embedded in $H$). Since $H$ has no cycle smaller than $2\pi$, an edge of $H$ not in $i_k(G_6)$ has an extremity in $T_\mathrm{u}^k$ and the other $C^k$. Furthermore, there is a unique way to realize the pairing. The pairing associated with $i_1(G_6)$ induces a pairing for $i_2(G_6)$ using  $\theta_0'$ and $\theta_0''$. By unicity  this pairing coincides with that associated with $i_2(G_6)$. This shows that $\theta_0$ extends to an automorphism of $H$.
\end{proof}

\begin{lemma}
The buildings  ${X_6^0}'$ and ${X_6^1}'$ associated to $X_6^0$ and $X_6^1$ are not isomorphic.
\end{lemma}

\begin{proof}
By a theorem of Tits \cite{tits-sphere},  the spheres of radius 2 in buildings of type $\tilde A_2$ and order 2 fall into two isomorphism classes, corresponding respectively to the building of $\GL_3(\QI_2)$ and $\GL_3(\FI_2((y)))$. It turns out (using for example  \cite[p. 580]{toulouse}) that  the sphere of radius 2 of ${X_6^0}'$ and ${X_6^1}'$ correspond to different isomorphism classes. In particular ${X_6^0}'$ and ${X_6^1}'$ are not isomorphic.  
\end{proof}

Combining the above lemmas, we have now proved Theorem \ref{th2} (modulo Theorem \ref{th:loccrit}).

\begin{remarks}\begin{enumerate}
\item Let ${V_6^1}'={X_6^1}'/{\G_6^1}'$ be the complex corresponding to the above unique extension of $(X_6^1,\G_6^1)$. One can show that the automorphism group of ${V_6^1}'$ is \emph{transitive on faces}.  
In particular, any two buildings with one chamber missing which admit an extension to $({X_6^1}', {\G_6^1}')$ are isomorphic.  
\item We saw in the spherical case, that degenerate buildings with chambers missing were not attached to a well-defined Coxeter diagram. Concerning degeneracy (=hyperbolicity) in the Euclidean case,  one can deform the metric on $V_0$ or $V_1$ to obtain a complex made of $(2,4,4)$ triangles. The resulting universal is also hyperbolic, in fact, it is equivariantly quasi-isometric to $X_0$ or $X_1$. But this is not a CAT(0) space as is easily checked (the link contains exactly one circle of length $3\pi\over 2$).  
\end{enumerate}
\end{remarks}

 As mentioned in the introduction, it would be interesting to classify Euclidean buildings with a single chamber missing that are hyperbolic. We observe that there is another remarkable complex that falls in this family and is associated to $G_6$:
 \begin{figure}[h]
\begin{tikzpicture}[y=.8cm,scale=.9]
\node at (-.9,0) {$V_6^3=$};
\draw[-latex] (1,0) -- node[anchor=west]{\tiny{1}} (0.5,0.866);
\draw[-latex] (0.5,0.866) -- node[anchor=east]{\tiny{2}} (0,0);
\draw[-latex] (1,0) -- node[anchor=west]{\tiny{3}} (0.5,-0.866);
\draw[-latex] (0.5,-0.866) -- node[anchor=east]{\tiny{3}} (0,0);
\shade[ball color=black] (0,0) circle (0.3ex);
\shade[ball color=black] (1,0) circle (0.3ex);
\shade[ball color=black] (0.5,0.866) circle (0.3ex);
\shade[ball color=black] (0.5,-0.866) circle (0.3ex);
\draw[-latex] (2.6,0) -- node[anchor=west]{\tiny{1}} (2.1,0.866);
\draw[-latex] (2.1,0.866) -- node[anchor=east]{\tiny{3}} (1.6,0);
\draw[-latex] (2.6,0) -- node[anchor=west]{\tiny{4}} (2.1,-0.866);
\draw[-latex] (2.1,-0.866) -- node[anchor=east]{\tiny{4}} (1.6,0);
\shade[ball color=black] (1.6,0) circle (0.3ex);
\shade[ball color=black] (2.6,0) circle (0.3ex);
\shade[ball color=black] (2.1,0.866) circle (0.3ex);
\shade[ball color=black] (2.1,-0.866) circle (0.3ex);
\draw[-latex] (4.2,0) -- node[anchor=west]{\tiny{1}} (3.7,0.866);
\draw[-latex] (3.7,0.866) -- node[anchor=east]{\tiny{4}} (3.2,0);
\draw[-latex] (4.2,0) -- node[anchor=west]{\tiny{2}} (3.7,-0.866);
\draw[-latex] (3.7,-0.866) -- node[anchor=east]{\tiny{2}} (3.2,0);
\shade[ball color=black] (3.2,0) circle (0.3ex);
\shade[ball color=black] (4.2,0) circle (0.3ex);
\shade[ball color=black] (3.7,0.866) circle (0.3ex);
\shade[ball color=black] (3.7,-0.866) circle (0.3ex);
\end{tikzpicture}
\end{figure}

\noindent It gives rise a third hyperbolic building with one chamber missing $(X_6^3,\G_6^3)$ that satisfy all assertions of Theorem \ref{th2}. One can show that in fact the corresponding building ${X_6^3}'$ is  isometric to ${X_6^1}'$. The abelianization group is $H_1(V_6^3,\ZI)\simeq \ZI\times \ZI/7\ZI$.

We will now study complexes associated with the other links, in particular $G_1$, $G_3$ and $G_4$, which also provide a few additional hyperbolic examples. For each of these 3 graphs, we give one example of a simplicial complex which is a CAT(0) space that is transitive on vertices and has the given link.   
The strategy to discover these complexes is to start from the link and guess what are the possible shapes that could be used to obtain the length spectrum of the link.  Then it remains to label the edges in order to have the correct link.

We start with $G_1$. The length spectrum is $\{({\pi\over 3}, 9),(\pi,3)\}$. Consider the following complex. 
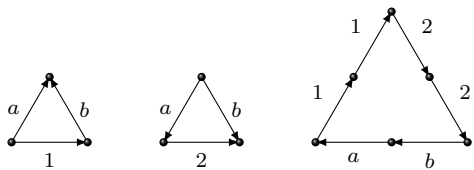
\begin{figure}[h]
\centering
\begin{tikzpicture}

\begin{scope}[xshift=-2cm,yshift=0cm]
\draw[-latex] (1,0) -- node[anchor=west]{\tiny{$b$}} (0.5,0.866);
\draw[latex-] (0.5,0.866) -- node[anchor=east]{\tiny{$a$}} (0,0);
\draw[-latex] (0,0) -- node[anchor=north]{\tiny{1}} (1,0);
\shade[ball color=black] (0,0) circle (0.3ex);
\shade[ball color=black] (1,0) circle (0.3ex);
\shade[ball color=black] (0.5,0.866) circle (0.3ex);
\end{scope}

\begin{scope}[xshift=0cm,yshift=0cm]
\draw[latex-] (1,0) -- node[anchor=west]{\tiny{$b$}} (0.5,0.866);
\draw[-latex] (0.5,0.866) -- node[anchor=east]{\tiny{$a$}} (0,0);
\draw[-latex] (0,0) -- node[anchor=north]{\tiny{2}} (1,0);
\shade[ball color=black] (0,0) circle (0.3ex);
\shade[ball color=black] (1,0) circle (0.3ex);
\shade[ball color=black] (0.5,0.866) circle (0.3ex);
\end{scope}

\begin{scope}[xshift=3cm]
\path (0:1cm) coordinate (P0);
\path (1*60:1cm) coordinate (P1);
\path (2*60:1cm) coordinate (P2);
\path (3*60:1cm) coordinate (P3);
\path (0,0) coordinate (P4);
\path (0,2*0.866) coordinate (P5);
\draw [latex-] (P0) -- node[anchor=south west]{\tiny{2}}  (P1);
\draw [-latex] (P5) -- node[anchor=south west]{\tiny{2}}  (P1);
\draw [latex-] (P5) -- node[anchor=south east]{\tiny{1}}  (P2);
\draw [-latex] (P3) -- node[anchor=south east]{\tiny{1}}  (P2);
\draw [-latex] (P4) -- node[anchor=north]{\tiny{$a$}}  (P3);
\draw [-latex] (P0) -- node[anchor=north]{\tiny{$b$}}  (P4);
\foreach \i in {0,...,5} \shade[ball color=black] (P\i) circle (0.3ex);
\end{scope}
\end{tikzpicture}
\caption{The complex $V_1$ with link $G_1$}\label{fig7}
\end{figure}

\noindent
It is built out of three equilateral triangles, one of them being scaled up (mind the opposite orientation of edges on the two smaller triangles). It can be shown that:

\begin{proposition}\label{prop- v1 hyperbolic} The universal cover $X_1=\tilde V_1$ of $V_1$ is hyperbolic. 
\end{proposition}

The group $\G_1$ admits the presentation
\[
\G_1=\langle s, t\mid ts^{-1}t^{-2}s=s^2  t^2\rangle
\]
and has $H_1(\G_1,\ZI)=\ZI$. The group $\G_1$ doesn't contain a copy of $\ZI^2$.

Consider now the graph $G_3$, which has length spectrum $\{ ({\pi\over 3}, 8),({2\pi\over 3},3),({4\pi\over 3},1)$. The following complex $V_3$ has link $G_3$ (note that the orientation is reversed on the first rhombus which leads to having ${4\pi\over 3}$ belonging to the spectrum).

\begin{figure}[h]
\centering
\begin{tikzpicture}

\begin{scope}[xshift=1cm,yshift=-.5cm]
\draw[-latex] (1,0) -- node[anchor=west]{\tiny{4}} (0.5,0.866);
\draw[-latex] (0.5,0.866) -- node[anchor=east]{\tiny{5}} (0,0);
\draw[-latex] (0,0) -- node[anchor=north]{\tiny{3}} (1,0);
\shade[ball color=black] (0,0) circle (0.3ex);
\shade[ball color=black] (1,0) circle (0.3ex);
\shade[ball color=black] (0.5,0.866) circle (0.3ex);
\end{scope}

\begin{scope}[xshift=-1cm,yshift=-.5cm]
\draw[-latex] (1,0) -- node[anchor=west]{\tiny{2}} (0.5,0.866);
\draw[-latex] (0.5,0.866) -- node[anchor=east]{\tiny{2}} (0,0);
\draw[-latex] (0,0) -- node[anchor=north]{\tiny{5}} (1,0);
\shade[ball color=black] (0,0) circle (0.3ex);
\shade[ball color=black] (1,0) circle (0.3ex);
\shade[ball color=black] (0.5,0.866) circle (0.3ex);
\end{scope}

\begin{scope}[xshift=-3cm]
\draw[-latex] (1,0) -- node[anchor=west]{\tiny{2}} (0.5,0.866);
\draw[-latex] (0.5,0.866) -- node[anchor=east]{\tiny{1}} (0,0);
\draw[latex-] (0.5,-0.866) -- node[anchor=east]{\tiny{3}} (0,0);
\draw[latex-] (1,0) -- node[anchor=west]{\tiny{1}} (0.5,-0.866);
\shade[ball color=black] (0,0) circle (0.3ex);
\shade[ball color=black] (1,0) circle (0.3ex);
\shade[ball color=black] (0.5,0.866) circle (0.3ex);
\shade[ball color=black] (0.5,-0.866) circle (0.3ex);
\end{scope}

\begin{scope}[xshift=3cm]
\draw[-latex] (1,0) -- node[anchor=west]{\tiny{5}} (0.5,0.866);
\draw[-latex] (0.5,0.866) -- node[anchor=east]{\tiny{4}} (0,0);
\draw[-latex] (0.5,-0.866) -- node[anchor=east]{\tiny{3}} (0,0);
\draw[-latex] (1,0) -- node[anchor=west]{\tiny{4}} (0.5,-0.866);
\shade[ball color=black] (0,0) circle (0.3ex);
\shade[ball color=black] (1,0) circle (0.3ex);
\shade[ball color=black] (0.5,0.866) circle (0.3ex);
\shade[ball color=black] (0.5,-0.866) circle (0.3ex);
\end{scope}

\end{tikzpicture}
\caption{The complex $V_3$ with link $G_3$}\label{fig4}
\end{figure}
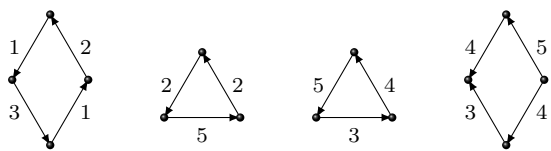

The group $\G_3$ is also a 1-relator group with presentation
\[
\G_3=\langle s, t\mid  s  t  s  t^2  s  t  s  =t^2stst^{2}\rangle
\]
and has $H_1(\G_3,\ZI)=\ZI$. However, contrary to the case of $V_1$, we have

\begin{proposition}
The group $\G_3$ contains a copy of $\ZI^2$.
\end{proposition}

\begin{proof}
Let $x=sts$ so that the defining relation of $\G_3$ can be rewritten as $xt^2x=t^2xt^2$. We have
\begin{align*}
x(t^2x)^3&=xt^2xt^2xt^2x\\
&=(xt^2x)(t^2xt^2)x\\
&=(t^2xt^2)(xt^2x)x\\
&=(t^2x)^3x.
\end{align*}
It follows that the subgroup of $\G_3$ generated by $sts$ and $t^2stst^2stst^2sts$ is isomorphic to $\ZI^2$.
\end{proof}

Our last complex is associated with $G_4$ which has length spectrum $\{ ({\pi\over 3}, 7),({2\pi\over 3},4),(\pi,1)\}$.

\begin{figure}[h]
\centering
\begin{tikzpicture}

\begin{scope}[xshift=0cm,yshift=0cm]
\draw[-latex] (1,0) -- node[anchor=west]{\tiny{4}} (0.5,0.866);
\draw[-latex] (0.5,0.866) -- node[anchor=east]{\tiny{3}} (0,0);
\draw[-latex] (0,0) -- node[anchor=north]{\tiny{2}} (1,0);
\shade[ball color=black] (0,0) circle (0.3ex);
\shade[ball color=black] (1,0) circle (0.3ex);
\shade[ball color=black] (0.5,0.866) circle (0.3ex);
\end{scope}

\begin{scope}[xshift=2cm,yshift=.5cm]
\draw[-latex] (1,0) -- node[anchor=west]{\tiny{4}} (0.5,0.866);
\draw[-latex] (0.5,0.866) -- node[anchor=east]{\tiny{2}} (0,0);
\draw[-latex] (0.5,-0.866) -- node[anchor=east]{\tiny{4}} (0,0);
\draw[-latex] (1,0) -- node[anchor=west]{\tiny{3}} (0.5,-0.866);
\shade[ball color=black] (0,0) circle (0.3ex);
\shade[ball color=black] (1,0) circle (0.3ex);
\shade[ball color=black] (0.5,0.866) circle (0.3ex);
\shade[ball color=black] (0.5,-0.866) circle (0.3ex);
\end{scope}

\begin{scope}[xshift=-2cm,yshift=0cm]
\path (0:1cm) coordinate (P0);
\path (1*60:1cm) coordinate (P1);
\path (2*60:1cm) coordinate (P2);
\path (3*60:1cm) coordinate (P3);
\path (0,0) coordinate (P4);
\draw [-latex] (P1) -- node[anchor=west]{\tiny{3}}  (P0);
\draw [-latex] (P1) -- node[anchor=south]{\tiny{1}}  (P2);
\draw [-latex] (P3) -- node[anchor=east]{\tiny{2}}  (P2);
\draw [-latex] (P4) -- node[anchor=north]{\tiny{1}}  (P3);
\draw [-latex] (P0) -- node[anchor=north]{\tiny{1}}  (P4);
\foreach \i in {0,...,4} \shade[ball color=black] (P\i) circle (0.3ex);
\end{scope}
\end{tikzpicture}
\caption{The complex $V_4$ with link $G_4$}\label{fig5}
\end{figure}
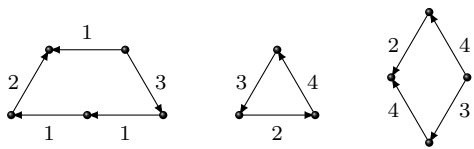

The group $\G_4$ is also a 1-relator group with presentation
\[
\G_4=\langle s, t\mid  s^{-1} t  s^2  t  s^{-2}  t^{-1}  s  t  s^{-1}  t  s^2 t^{-1}= e\rangle
\]
and has $H_1(\G_4,\ZI)=\ZI$

\begin{proposition} The universal cover $X_4=\tilde V_4$ of $V_4$ is hyperbolic. 
\end{proposition}

We leave the proof of this proposition as well as that of Proposition \ref{prop- v1 hyperbolic} to the reader.

Remain the two graphs $G_2$ and $G_5$. The case of $G_2$ is similar and will be treated in a forthcoming section. It is tempting to try the same approach on the last graph $G_5$ (which has the same length spectrum as $G_6$), but this doesn't work. We refer to Section \ref{Sectinvert} for a precise description of the situation in this case.

\section{The local rank functional on metric graphs}\label{CATrank}

We introduce a notion of ``CAT(0) rank" for an arbitrary metric graph of girth at least $2\pi$.  This generalizes and subsumes the qualitative definitions of a similar concept of intermediate local rank introduced in \cite{rd}.

Let $L$ be a locally finite metric graph, denote by $S$ the set of singular vertices of $L$ (where we call \emph{singular} a vertex of valency at least 3) and let 
\[
\Phi_L=\{\alpha : [0,\pi] \inj L,~\alpha(0)\in S\} 
\]
be the set of paths of length $\pi$ in $L$ starting from a singular vertex, where two paths are considered distinct if they have distinct images in $L$.
We assume that the girth of $L$ is at least $2\pi$.  
We also assume that $\Phi_L$ is nonempty.  Observe that if $\Phi_L$ is empty, then $L$ is ``degenerate":  it is either a union of circles or segments (if $S$ is empty), or trees of radius $<\pi$ around some point of $S$ (if $S$ is nonempty). We sometimes (improperly) refer to $\Phi_L$ as the \emph{root system} for $L$, and to its elements as \emph{roots}.

For a vertex $v\in S$ denote by $q_v$ its \emph{order} in $L$, that is, $q_v=\val_L(v)-1$ where $\val_L$ is the valency. For $\alpha\in \Phi_L$ we write $q_\alpha=q_{\alpha(0)}$.

\begin{definition}
We call \emph{rank} of an element $\alpha\in \Phi_L$  the number 
\[
\rk(\alpha)=1+{N(\alpha)\over q_\alpha}
\] 
where $N(\alpha)$ is the number of path of length $\pi$ in $L$ distinct from $\alpha$ whose extremity coincide with that of $\alpha$. Formally:
\[
N(\alpha)=|\{\beta\in \Phi_L\mid \alpha\neq \beta, \alpha(0)=\beta(0), \alpha(\pi)=\beta(\pi)\}|. 
\] 
\end{definition}

Note that:

\begin{fact}
For every $\alpha\in \Phi_L$ we have $N(\alpha)\leq q_\alpha$. 
\end{fact}

\begin{definition}[CAT(0) rank of a metric graph]\label{Def - local rank} Let $L$ be a finite metric graph of girth at least $2\pi$ and assume that $\Phi_L$ is a nonempty finite set.  We call \emph{CAT(0) rank of $L$}, and denote  $\rk(L)$, the mean value
\[
\rk(L)={1\over |\Phi_L|}\sum_{\alpha\in \Phi_L} \rk(\alpha)
\] 
so that $\rk(L)\in [1,2]$, and  $L$ has rank 1 (resp. 2) if and only if  every root $\alpha\in \Phi_L$ has rank 1 (resp. 2). 
\end{definition}

We have the following fact.

\begin{proposition}
Let $L$ be a connected finite metric graph without leaf of girth at least $2\pi$. Then $L$ has rank 2 if and only if it is a spherical building. 
\end{proposition}

Let now $X$ denote a polyhedral complex of dimension 2 without boundary (i.e.\ every point of the 1-skeleton belongs to at least two 2-cells). Assume that $X$ is piecewise linear, i.e., that every 2-cell is endowed with a fixed (pairwise compatible) Euclidean metric with linear boundary. For each $x\in X$, the link $L_x$ of $X$ is the sphere of sufficiently small radius around $x$, endowed with the angular metric (the resulting graph is independent of the choice of the radius, provided it is small enough).  We recall that $X$ is a CAT(0) space if and only if for every $x\in X$, the link $L_x$ has girth at least $2\pi$ (see \cite[Chapter II.5]{BH}).  

A point $x$ of $X$ is said to be regular if its link is isometric to a circle of length $2\pi$. It is said to be singular otherwise. The set of singular points is included in the vertex set of $X$. We will be distinguishing between several types of singular vertices. A \emph{smooth singularity} in $X$ is a vertex whose link is a circle of length $>2\pi$. An \emph{open book singularity} is a singularity whose link is isometric to the direct product of a finite set with a two elements set, where all edges have length $\pi/4$.   Otherwise, we say that $X$ has an \emph{essential singularity} at this vertex.   
The complex $X$ itself is said to be \emph{singular} if it contains an essential singularity.  We also say that a vertex $x$ of $X$ is a \emph{removable singularity} of $X$ if it is a regular point of $X$  or an open book singularity. 

We call  \emph{link} of a singular complex $X$ the union $L_X$ taken over all non removable singularity $x$ of $X$ of the links $L_x$ at $x$. If $x$ is smooth, then the root system $\Phi_x$ at $x$ is defined to be the set of all path of length $\pi$ in $L_x$ issued from vertices. If $x$ is essential, then we define $\Phi_x$ to be the root system of $L_x$ is $x$.  The \emph{root system} $\Phi_X$ of $X$ is the union of the sets $\Phi_x$ taken over all non removable singularities $x\in X$. This is a metric  graph, usually disconnected, of girth at least $2\pi$. Elements in $\Phi_x$ are called the roots at $x$. Every root at a smooth singularity has rank 1.

Let $\G$ be a group of isometries of $X$. Then $\G$ acts by isometries on $L_X$ and  on the root system $\Phi_X$ of $X$. Furthermore the rank functional $\rk$ on $\Phi_X$ is invariant under this action.

\begin{definition}[Local rank of $(X,\G)$]\label{dfnlocrk} Let $X$ be a singular CAT(0) complex of dimension 2 and let $\G$ be a group of isometries of $X$ with $\Phi_X/\G$ finite.  The \emph{local rank} of $(X,\G)$ is defined to be the rank of the link of $(X,\G)$, namely:  
\[
\rk(X,\G)={1\over |\Phi_X/\G|} \sum_{\alpha\in \Phi_X/\G} \rk(\alpha).
\]
\end{definition}

\begin{remark}
The idea in rank interpolation as studied in this paper is to look for complexes (once some class of ``ambient" singular CAT(0) polyhedral complexes has been fixed, for example the triangle complexes, which are built out of equilateral triangles) whose local rank as close to 2 as one wishes, without being equal to 2 itself. This forces the proportion of vertices isomorphic to buildings in $X$ (relative to $\G$) to converge to 1;   examples (of a random nature) of such triangle complexes are given in \cite{random}. By a result of Tits, the assumption that $\rk(X,\G)=2$ (for $X$ a singular triangle complex) implies that $X$ is a thick building of type $\tilde A_2$ (in which case the group $\G$ satisfies the property T of Kazhdan).
\end{remark}

It is also possible to define $\ell^p$ versions of the local rank of a metric graph $L$, in the usual way, say $\rk_p(L)$ for $p\in [1,\infty]$, where in particular,
$\rk_\infty(L)=\sup_{\alpha\in \Phi_L} \rk(\alpha)$ and $\rk_\infty(X)=\sup_{x\in X^{(0)}} \rk_\infty(L_x)$ for any metric polyhedral complex $X$. We also have the following related notion   (compare Lemma \ref{lem:isfl}), which is a local version of the isolated flat property.

\begin{definition}\label{rem:isfl}
We say that $L$ has  rank $\rk_\infty(L)\leq 1^+$ if for every $\alpha\in \Phi_L$ we have $N(\alpha)\leq 1$.
\end{definition}

We note that if $L$ has constant $q_\alpha=q$, then $\rk_\infty(L)\leq 1^+$ if and only if $\rk_\infty(L)\leq 1+{1\over q}$ in the sense above. We extend  $\rk_\infty(X)=\sup_{x\in X^{(0)}} \rk_\infty(L_x)$ to the extra symbol $1^+$ and say that $X$ has local rank $1^+$ if $\rk_\infty(X)\leq 1^+$.

Let us illustrate the above notions by  computing precise values of the rank in instances of interest to the present paper.

\begin{proposition}
Let $G$ be a spherical  building of type $A_2$ and  of order $q$ with one missing chamber. Then 
\[
\rk(G)= 2- {2\over (q+1)(q^2+q+1)-3}
\]
for $q\geq 3$, while 
\[
\rk(G)= 2-{1\over 8}= 1.875
\]
if $q=2$.
\end{proposition}

\begin{proof}
Let $\alpha$ be a root. Then either $\rk(\alpha)=2$ or $\rk(\alpha)=2-{1\over q}$. The latter case corresponds to $q_\alpha=q$ and $N(\alpha)=q-1$ and can be divided into two subcases: the missing chamber can be at distance 1 or 2 from $\alpha(0)$. Both contribute to $2q^3$ roots, hence $4q^3$ roots of rank $2-{1\over q}$.

Let us compute the number of roots of rank 2. For such a root $\alpha$, we have either $q_\alpha=q$ or $q_\alpha=q-1$. The latter case contributes to $2q^3$ roots if $q>2$ and to none if $q=2$. Let us now assume that $q_\alpha=q$. There are again two subcases: the missing chamber can be at distance 1 or 2 from $\alpha(0)$.  

The first subcase ramifies into 2 subcases, depending on whether the closest point of the root to the missing chamber is $\alpha(0)$ or  $\alpha(\pi/3)$. In the case of $\alpha(0)$, this contributes to $2q^3(q-1)$ roots, for $\alpha(\pi)$ does not belong to the 6-cycle containing the missing chamber and  $\alpha([0,2\pi/3]$).  In the case of $\alpha(\pi/3)$, this contributes to $2q^2(q-1)$ as is easily seen.

The second subcase ramifies into 4 subcases, depending on whether the closest point of the root to the missing chamber is $\alpha(\pi/3)$ or $\alpha(2\pi/3)$ (at distance 1 and 0 respectively), or whether the missing chamber is opposite to $\alpha([0,\pi/3])$  or to $\alpha([2\pi/3,\pi/2])$ (we recall that opposite means opposite in some apartment).  Analyzing the possibilities gives respectively  $2q^3(q-1)$, $2q^2(q-1)$, $2q^3(q-1)$, and $2q^4(q-1)$ roots. 

Thus  we find 
$4q^3$ roots of rank $2-{1\over q}$ and 
\[
2q^2(q+2(q-1)+3q(q-1)+q^2(q-1))=2q^2(q^3+2q^2-2)
\] 
roots of rank 2 if $q>2$. The rank of $G$ is
\[
\rk(G)={2q(2-{1\over q})+2(q^3+2q^2-2)\over  q^3+2q^2+2q-2}=2-{2\over q^3+2q^2+2q-2}.
\]

If $q=2$, then we have 32 roots of rank $3\over 2$ and 96 roots of rank 2 and the rank is $15/8$. 

This proves the proposition.  
\end{proof}

In particular, the rank converges to 2 as $q\to \infty$. 

\medskip

It is also instructive to examine the rank of the six spherical buildings of Proposition \ref{fact6prime}:

\begin{proposition}\label{rank3}
The rank of spherical $A_2$ buildings  of order 2 with three chambers missing is given in the following table:
\begin{table}[ht]
\centering
\begin{tabular}{c@{\hspace{.7cm}} c@{\hspace{.6cm}} c@{\hspace{.6cm}} c@{\hspace{.6cm}}c@{\hspace{.6cm}}c@{\hspace{.6cm}}c@{\hspace{.6cm}}}
& $G_1$&$G_2$&$G_3$&$G_4$&$G_5$&$G_6$\\
\\
$\rk(G_i)=$ & $18\over 11$&$13\over 8$&$105\over 64$&$49\over 31$&$3\over 2$&$3\over 2$\\
\\
$\rk(G_i)\approx$ & $1.636$&$1.625$&$1.640$&$1.58$&$1.5$&$1.5$\\
\end{tabular}
\end{table}
\end{proposition}

\medskip

This can be proved by a direct (but tedious) computation. In fact we have the following classification of roots:\\
\begin{table}[ht]
\centering
\begin{tabular}{c@{\hspace{.7cm}} | c@{\hspace{.6cm}} c@{\hspace{.6cm}} c@{\hspace{.6cm}}c@{\hspace{.6cm}}c@{\hspace{.6cm}}c@{\hspace{.6cm}}}
Roots $\backslash$ Buildings & $G_1$&$G_2$&$G_3$&$G_4$&$G_5$&$G_6$\\
\hline
rank 1 & $6$&$8$&$6$&$8$&$0$&$6$\\
\hline
rank 3/2 & $36$&$32$&$34$&$36$&$60$&$48$\\
\hline
rank 2 & $24$&$24$&$24$&$18$&$0$&$6$\\
\hline
$|\Aut|$ & $6$&$2$&$2$&$2$&$12$&$6$\\
\hline
\end{tabular}
\end{table}

We note that the inverse pyramid $G_5$ has the distinctive feature that all of its roots are of rank $3\over 2$. The table also shows the reason why we chose $G_6$ when looking for transitive Euclidean buildings with missing chambers that are hyperbolic (compare Fact \ref{lem:isfl}). A better choice \emph{a priori} could have been $G_5$---this is studied in full details in a later section. The computation of the automorphism groups is an easy exercise (the case of $G_5$ is solved in the introduction). As a corollary of this analysis we observe that:

\begin{lemma}\label{L:emma - aut-gi-fixes tripod}
Every automorphism of $G_1,\ldots, G_6$ which fixes a vertex and its three adjacent edges is trivial.
\end{lemma}

\section{Local-to-global type results for buildings with chambers missing}\label{locglob}

For spaces of intermediate rank there is no general local-to-global result allowing to control the mesoscopic or the asymptotic rank in terms of the local rank.  For hyperbolic spaces or buildings, we have the following well-known results (see for example \cite{BH} and \cite{Ronan}).

\begin{theorem}[Local-to-global]\label{thm:titslocal} Let $X$ be a CAT(0) simplicial complex of dimension 2.
\begin{itemize}
\item If $X$ is locally hyperbolic (which is equivalent to $X$ having local rank 1), then $X$ is hyperbolic. 
\item If all  links of $X$ are spherical buildings (which is equivalent to $X$ having local rank 2), then $X$ is a building (the converse also holds).   
\end{itemize}
\end{theorem}

The case of isolated flats is similar to that of hyperbolic spaces. 
Recall that a CAT(0) complex of dimension 2 has  \emph{isolated flats} if for every compact set $K$ of $X$, the number of flats of $X$ intersecting $K$ is finite (see \cite{HK}). We record the following fact for further reference (our notation $\rk_\infty(X)\leq 1^+$ is explained in Definition \ref{rem:isfl}).

\begin{lemma}\label{lem:isfl}
Let  $X$ be a CAT(0) simplicial complex of dimension 2 with local rank 
 $\rk_\infty(X)\leq 1^+$. Then $X$ has isolated flats.
\end{lemma}

\begin{proof}
If a space does not have isolated flats, then by Wise's criterion it contains a flat triplane (see \cite{HK}). A fortiori, some link $L$ at a vertex on the singular set of this triplane contains an $\alpha\in \Phi_L$ with $N_L(\alpha)> 1$, contradicting the assumption.
\end{proof}

In this section we study the passage from local to global for buildings with chambers missing. We give first two representative examples, and then derive a criterion (Theorem \ref{th:loccrit}) for a space to be a building with missing chamber of type $\tilde A_2$.

The first example shows that mesoscopic rank phenomena  can occur when a single chamber is missing and the acting group is transitive on vertices. 

\begin{proposition}\label{prop83} There exists a Euclidean building $(X_2,\G_2)$ with exactly one chamber missing  which is of (exponential) mesoscopic rank. The group $\G_2$ acts transitively on the vertices of $X_2$, contains $\ZI^2$, and the link of $X_2$ is isometric to $G_2$.
\end{proposition}

Mesoscopic rank is defined in \cite{rd}, and is meant to recognize spaces that contain balls of arbitrary large radius which are flat (isometric to balls in $\RI^2$) but \emph{cannot be embedded} into any flat $\simeq \RI^2$ of $X$. 
Other groups with this property include some groups of rank \sq\ (see \cite{rd}), or the braid group $B_4$ on 4 strings (see \cite{b4}).
 
To prove Proposition \ref{prop83}, we consider the complex $V_2$ built out of two equilateral triangles and a parallelogram of size $2\times1$ as follows:
\begin{figure}[h]
\centering
\begin{tikzpicture}

\begin{scope}[xshift=-2cm,yshift=0cm]
\draw[-latex] (1,0) -- node[anchor=west]{\tiny{1}} (0.5,0.866);
\draw[-latex] (0.5,0.866) -- node[anchor=east]{\tiny{4}} (0,0);
\draw[-latex] (0,0) -- node[anchor=north]{\tiny{2}} (1,0);
\shade[ball color=black] (0,0) circle (0.3ex);
\shade[ball color=black] (1,0) circle (0.3ex);
\shade[ball color=black] (0.5,0.866) circle (0.3ex);
\end{scope}

\begin{scope}[xshift=0cm,yshift=0cm]
\draw[-latex] (1,0) -- node[anchor=west]{\tiny{3}} (0.5,0.866);
\draw[-latex] (0.5,0.866) -- node[anchor=east]{\tiny{4}} (0,0);
\draw[-latex] (0,0) -- node[anchor=north]{\tiny{1}} (1,0);
\shade[ball color=black] (0,0) circle (0.3ex);
\shade[ball color=black] (1,0) circle (0.3ex);
\shade[ball color=black] (0.5,0.866) circle (0.3ex);
\end{scope}

\begin{scope}[xshift=3cm]
\path (0:1cm) coordinate (P0);
\path (1*60:1cm) coordinate (P1);
\path (2*60:1cm) coordinate (P2);
\path (3*60:1cm) coordinate (P3);
\path (0,0) coordinate (P4);
\path (1.5,0.866) coordinate (P5);
\draw [-latex] (P0) -- node[anchor=west]{\tiny{1}}  (P5);
\draw [-latex] (P5) -- node[anchor=south]{\tiny{2}}  (P1);
\draw [-latex] (P1) -- node[anchor=south]{\tiny{3}}  (P2);
\draw [-latex] (P3) -- node[anchor=east]{\tiny{4}}  (P2);
\draw [-latex] (P4) -- node[anchor=north]{\tiny{2}}  (P3);
\draw [-latex] (P0) -- node[anchor=north]{\tiny{3}}  (P4);
\foreach \i in {0,...,5} \shade[ball color=black] (P\i) circle (0.3ex);
\end{scope}
\end{tikzpicture}
\caption{A complex $V_2$ with link $G_2$}\label{fig7}
\end{figure}
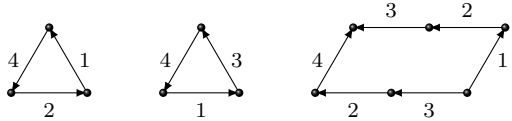

The fundamental group $\G_2=\pi_1(V_2)$ is defined by the presentation: 
\[
\G_2=\langle u,v\mid u^{-2}vu^2v=v^2uv^{-2}\rangle
\]

The proof that $X_2$ is of mesoscopic rank can be done in a similar fashion to that given in \cite{rd}  and \cite{b4}. We will omit it here since it doesn't bring any new idea. The other statement can be proved by using techniques developed in Section \ref{sec2}. It can be seen proved directly that $(X_2,\G_2)$ is a building with chamber missing.

\begin{proposition}\label{No local to global} There exists a polyhedral complex $X_3^6$ endowed with a free action of a group $\G_3^6$ which transitive on vertices, that is \emph{not} a building with missing chamber, but all of whose links are buildings with chambers missing (all links are isomorphic to $G_6$). 
\end{proposition}

We give an example which may be compared to the constructions of Section \ref{sec2} (see also Remark \ref{rem:b4}), which \emph{were} leading to building with chambers missing. Consider the complex $V_6^3$ defined by:
  \begin{figure}[h]
  \centering
\begin{tikzpicture}[y=.8cm]
\node at (-.7,0) {$V_6^3=$};
\draw[-latex] (1,0) -- node[anchor=west]{\tiny{1}} (0.5,0.866);
\draw[-latex] (0.5,0.866) -- node[anchor=east]{\tiny{2}} (0,0);
\draw[-latex] (1,0) -- node[anchor=west]{\tiny{2}} (0.5,-0.866);
\draw[-latex] (0.5,-0.866) -- node[anchor=east]{\tiny{3}} (0,0);
\shade[ball color=black] (0,0) circle (0.3ex);
\shade[ball color=black] (1,0) circle (0.3ex);
\shade[ball color=black] (0.5,0.866) circle (0.3ex);
\shade[ball color=black] (0.5,-0.866) circle (0.3ex);
\draw[-latex] (2.6,0) -- node[anchor=west]{\tiny{1}} (2.1,0.866);
\draw[-latex] (2.1,0.866) -- node[anchor=east]{\tiny{3}} (1.6,0);
\draw[-latex] (2.6,0) -- node[anchor=west]{\tiny{3}} (2.1,-0.866);
\draw[-latex] (2.1,-0.866) -- node[anchor=east]{\tiny{4}} (1.6,0);
\shade[ball color=black] (1.6,0) circle (0.3ex);
\shade[ball color=black] (2.6,0) circle (0.3ex);
\shade[ball color=black] (2.1,0.866) circle (0.3ex);
\shade[ball color=black] (2.1,-0.866) circle (0.3ex);
\draw[-latex] (4.2,0) -- node[anchor=west]{\tiny{1}} (3.7,0.866);
\draw[-latex] (3.7,0.866) -- node[anchor=east]{\tiny{4}} (3.2,0);
\draw[-latex] (4.2,0) -- node[anchor=west]{\tiny{4}} (3.7,-0.866);
\draw[-latex] (3.7,-0.866) -- node[anchor=east]{\tiny{2}} (3.2,0);
\shade[ball color=black] (3.2,0) circle (0.3ex);
\shade[ball color=black] (4.2,0) circle (0.3ex);
\shade[ball color=black] (3.7,0.866) circle (0.3ex);
\shade[ball color=black] (3.7,-0.866) circle (0.3ex);
\end{tikzpicture}
\end{figure}

\noindent The fundamental group of $V_6^3$ is  (denoting $b^a=a b a^{-1}$):
\[
\G=\langle u,v\mid v  v^u  u^{v^{-1}}= v^uu^v  u \rangle
\]
with $H^1(V_6^3,\ZI)=\ZI$.
Links in $X_6^3=\tilde V_6^3$ are isometric to $G_6$, and $\G_6^3$ is transitive on vertices. That $(X_6^3,\G_6^3)$ is not a building with chambers missing can be checked using Theorem \ref{th:loccrit} below. The extension invariant of $(X_6^3,\G_6^3)$, in the sense of Definition \ref{def:EGI}, is as follows.
\begin{figure}[h]
\centering
\begin{tikzpicture}
\node at (-1,0) {$O_{X_6^3,\G_6^3}=$};
\coordinate [label=left: {}] (A1) at (0,0);
\coordinate [label=left: {}] (A2) at (1,0);
\coordinate [label=left: {}] (A3) at (1.5,0);
\coordinate [label=left: {}] (A4) at (2.5,0);
\coordinate [label=left: {}] (A5) at (3,0);
\coordinate [label=left: {}] (A6) at (4,0);

\draw[dashed] (A1) to [bend right] (A2);
\draw[-] (A1) to [bend left] (A2);
\draw[dashed] (A3) to [bend right] (A4);
\draw[-] (A3) to [bend left] (A4);
\draw[dashed] (A5) to [bend right] (A6);
\draw[-] (A5) to [bend left] (A6);

\shade[ball color=black] (A1) circle (0.3ex);
\shade[ball color=black] (A2) circle (0.3ex);
\shade[ball color=black] (A3) circle (0.3ex);
\shade[ball color=black] (A4) circle (0.3ex);
\shade[ball color=black] (A5) circle (0.3ex);
\shade[ball color=black] (A6) circle (0.3ex);

\end{tikzpicture}
\end{figure}

Having understood the above examples, we now turn to the main result of this section, which provides a local test  for deciding when the space under consideration is a building with chambers missing. Before stating our criterion, we need to define the \emph{order} of a building with chambers missing of type $\tilde A_2$. Recall that if $X$ is a Euclidean building of type $\tilde A_2$, then the order of $X$ is the number of faces incident to each edge, minus 1. One way to extend this definition to buildings with chambers missing is as follows.

Given a finite connected graph $G$, we call \emph{projective order of $G$} the number $q_*(G)$ defined as 
\[
q_{*}(G)={1\over 2} \sqrt{2 |G|_0-3}-1
\]
where $|G|_0$ is the number of vertices.

In other words,  $q_*$ satisfies $2(q_*^2+q_*+1)=|G|_0$. For the spherical building of a projective plane over the field $\F_q$, we have $q_*=q$, which is an integer. For spherical buildings with chambers missing, we have:

\begin{lemma} 
Let $G$ is a non-degenerate finite building with chambers missing of type $A_2$ and $G\leadsto H$ be an extension. Then $q_*(G)$ is an integer whose value is the order of $H$.
\end{lemma}

Let $X$ be a simplicial complex of dimension 2. We call \emph{projective order} of $X$ the number 
\[
q_*(X)=\max_{x\in X^{(0)}}  q_*(L_x),
\] 
which is finite if $X$ is uniformly locally finite. If $X$ is a Euclidean building of type $\tilde A_2$, then $q_*(X)$ coincide with the classical definition recalled above. If $X$ is a complex of rank \sq, then $q_*(X)\approx 2.19$.

We now assume until the end of this section that $X$ is simply connected and endowed with a free action of a countable group $\G$ with compact quotient, and denote by $q_*$ the projective order of $X$. 

\begin{definition}\label{def:EGI} Assume that every link of $X$ is a non-degenerate spherical building with chambers missing of type $A_2$ (in particular $q_*(X)$ is an integer). We call \emph{extension invariant of $(X,\G)$} the finite labelled graph $O_{X,\G}$ defined as follows. 
\begin{enumerate}
\item Let $E_0$ be the set of $\G$-orbit of (non-oriented) edges of $X$ which are incident to at most $q_*(X)$ chambers in $X$. (We think of elements in $E_0$ as middle thirds segments in their corresponding edge in $X$.)  We let $V$ be the  set of (equivariant) vertices corresponding to $E_0$ with natural boundary map, and associate an integer label equal to $q_*(X) - v+1\geq 1$ to edges in $E_0$, where $v$ is the chamber valency of the given edge orbit in $X$.

\item For any two distinct vertices $u,v\in V$ correspond to the same $\G$ orbit, so that $u$ and $v$ may be viewed as vertices of some link $L$ of $X$,  we associate an edge between $u$ and $v$  if and only at least 5 edges separate $u$ from $v$ in $L$. This set of edges is denotes $E_1$. 
\end{enumerate}
The extension invariant $O_{X,\G}$ is the finite  graph $(V,E_0\cup E_1)$. Edges in $E_0$ (resp. $E_1$) are said to be of type 0 (type 1). We omit the labeling of edges of $E_0$ when the label is 1, and draw a double edge when the label is 2. 
\end{definition}

\begin{example}
The extension invariant of the complex $(X,\G)$ described in Section \ref{Sectinvert}, Fig.\ \ref{fig5}, is given by
\begin{figure}[h]
\centering
\begin{tikzpicture}
\draw[-] (0,1) -- (0,0);
\draw[-] (0.5,1) -- (0.5,0);
\draw[dashed] (0,1) -- (.5,1);
\draw[dashed] (0,0) -- (0.5,0);

\draw[-] (1,1) -- (1,0);
\draw[-] (1.5,1) -- (1.5,0);
\draw[-] (2,1) -- (2.5,1);
\draw[-] (2,0) -- (2.5,0);

\draw[dashed] (1.5,1) -- (2,1);
\draw[dashed] (1.5,0) -- (2,0);

\draw[dashed] (1,1) to [bend left] (2.5,1);
\draw[dashed] (1,0) to [bend right] (2.5,0);

\shade[ball color=black] (0,0) circle (0.3ex);
\shade[ball color=black] (0.5,0) circle (0.3ex);
\shade[ball color=black] (1,0) circle (0.3ex);
\shade[ball color=black] (1.5,0) circle (0.3ex);
\shade[ball color=black] (2,0) circle (0.3ex);
\shade[ball color=black] (2.5,0) circle (0.3ex);

\shade[ball color=black] (0,1) circle (0.3ex);
\shade[ball color=black] (0.5,1) circle (0.3ex);
\shade[ball color=black] (1,1) circle (0.3ex);
\shade[ball color=black] (1.5,1) circle (0.3ex);
\shade[ball color=black] (2,1) circle (0.3ex);
\shade[ball color=black] (2.5,1) circle (0.3ex);
\end{tikzpicture}
\end{figure}

\noindent Therefore by Theorem \ref{th:loccrit} this complex is not a building with chambers missing. The two orbits of vertices in $X$ corresponds to the two horizontal sets of 6 vertices. Edges of type 1 are dashed.
\end{example}

We say that a path (possibly non injective on edges) in $O_{(X,\G)}$ is \emph{alternating} if two consecutive edges are of different type. We say that a family $\cF$ of subgraphs of $O_{X,\G}$ is  \emph{saturated} if the number of subgraphs of $\cF$ containing a given edge of type 0  is equal to the label of this edge, and that $\cF$ is \emph{ample} if for every link $L$ of $X$, the graph obtained by adding to $L$ all edges of type 1 of $\cF$ corresponding $L$ is ample (i.e.\ contains no cycle of length $\leq 5$).  

\begin{lemma}[Local criterion for $\tilde A_2$ buildings with chambers missing]\label{th:loccrit}
Let $(X,\G)$ be as above and assume that every link of $X$ is a spherical building with chambers missing of type $A_2$. Then $(X,\G)$ is a Euclidean building with chambers missing if and only if there exists a saturated ample family of alternating 6-cycles in the extension invariant  $O_{X,\G}$ of $(X,\G)$. 
\end{lemma}

\begin{proof}
Assume first that $(X,\G)$ is a building with chambers missing and choose an extension $(X,\G)\leadsto (X',\G')$. 
If $C_1,\ldots, C_n$ denote the corresponding family of chambers missing, then to $C_i$ is associated an alternating 6-cycle of $O_{X,\G}$: each edge of $C_i$ corresponds to an edge of type 0 in $O_{X,\G}$, and each angle of $C_i$ corresponds to an edge of type 1 in $O_{X,\G}$.  It is easy to check that all conditions are satisfied. We note that the 6-cycles in $O_{X,\G}$ can be of two types: either they are injective, or they contain exactly two  edges of $O_{X,\G}$ of type 0.

We prove the converse. Let $\cF$ be a saturated family of alternating 6-cycles in $O_{X,\G}$. For each vertex $v$ in $X$, consider the graph $L'_v$ which is obtained from the link  $L_v$ of $v$ by adding all edges of type 1 of $\cF$ which correspond to $v$. 

\begin{claim} $L_v'$ is a building of type $A_2$.
\end{claim} 
\begin{proof}
By the ampleness condition, the girth of $L_v'$ at least 6. By condition (1) in \ref{def:EGI} and since $\cF$ is saturated, the valency at vertices in $L_v'$  is constant equal to $q_*(X)+1$, 
and  
\[
|L_v'|_0=|L_v|_0=2(q_*(X)^2+q_*(X)+1).
\] 
This condition is well known to characterize incidence graphs of a (possibly exotic) projective plane.
\end{proof}

Thus $L_v\leadsto L_v'$ is an extension of $L_v$ into a spherical building. Furthermore, by definition of $O_{X,\G}$, the construction of $L_v'$ is equivariant with respect to the action of $\G$. 

Let us now construct a CW complex $V'$ containing $X/\G$ as follows. The 1-skeleton of $V'$ is that of $X/\G$, and faces are either faces of $V$ or new faces corresponding to the 6-cycles given by $\cF$. Let $X'=\tilde V'$ and $\G'=\pi_1(V')$. It is straightforward to check that the link at a vertex $v\in X'$ is $L_v'$. By Tits local criterion (see Theorem \ref{thm:titslocal}), it follows that $X'$ is a Euclidean building. Thus $(X,\G)\leadsto (X',\G')$ is an extension and $(X,\G)$ is a Euclidean building with chambers missing. 
\end{proof}

The above result further provides an upper-bound on the number of admissible extension, namely, the number of saturated family of alternating 6-cycle in the extension invariant. Although this is only a rough estimate for general buildings with chambers missing, this upper-bound is useful in two cases of interest to us: for concrete examples studied in this paper where a single is chamber missing (e.g. Corollary \ref{cor:unicity}),  and for generic buildings in certain random models (e.g. Corollary \ref{cor:unicity2}).  Unicity of extensions up to isomorphism can be seen as a  rigidity property of the corresponding building with a missing chamber.

\begin{theorem}\label{cor:unicity}
Let $(X,\G)$ be a building of type $\tilde A_2$ with a single chamber missing. Then there is a unique extension $(X,\G)\leadsto (X',\G')$ into a building of type $\tilde A_2$. 
\end{theorem}

\begin{proof}
Let $O_{X,\G}$ be the extension invariant of $(X,\G)$. We assume that $O_{X,\G}$ contains at least one saturated family of alternating 6-cycle and show that this family is unique. Since a single chamber is missing, $O_{X,\G}$ contains at most 6 vertices, and edges of type 0 have label either  1 or 2. If there is an edge with label 2, then one checks that the extension invariant is of the form

\begin{figure}[h]
\centering
\begin{tikzpicture}
\coordinate [label=left: {}] (A1) at (0,1);
\coordinate [label=left: {}] (A2) at (0,0);
\coordinate [label=left: {}] (A3) at (1,0);
\coordinate [label=left: {}] (A4) at (1,1);

\draw[dashed] (A1) -- (A2);
\draw[dashed] (A2) to [bend right] (A3);
\draw[dashed] (A3) -- (A4);
\draw[double] (A2)  --  (A3);
\draw[-] (A1) --  (A4);

\shade[ball color=black] (A1) circle (0.3ex);
\shade[ball color=black] (A2) circle (0.3ex);
\shade[ball color=black] (A3) circle (0.3ex);
\shade[ball color=black] (A4) circle (0.3ex);

\end{tikzpicture}
\end{figure}

\noindent hence unicity in that case. Let us now assume that type 0 edges have label 1. If at least one of these type 0 edges  has extremities associated to different orbits of vertices of $X$, then there must be two of them and unicity is then readily checked. Thus, we now assume that all vertices of $O_{X,\G}$ correspond to a same link $L$ of $X$. The link $L$ is a spherical building with 3 missing chambers of order $q=q_*(X)$,
and all saturated family in the extension invariant consists of a single alternating cycle of length 6. We must show that there is only one such a family, i.e.\ one possible 6-cycle.

Let $C$ be such a cycle and denote by $E=\{e,f,g\}$ the three edges between (unsaturated) vertices of $L$ which corresponds to type 1 edges of $C$, and let $L\leadsto L'$ be the corresponding extension into a building of type $A_2$. 

Assume first  that $E$ is included is an apartment of $L'$. Then since $L'$ is a building the extremities of any two distinct edges of $E$ can be then joined by a path in $L$ of length at most 4. Hence $O_{X,\G}$ is reduced to $C$.

Otherwise, at least two edges of $E$, say $e$ and $f$, are at distance 2 from each other (write $|e-f|=2$ in that case). Let $A$ be an apartment of $L'$ containing $e$ and $f$.
If $|g-e|=2$ and $|g-f|=2$, then extremities of edges in $E$ are at distance at most 3 in $L$ and so $O_{X,\G}$ is also reduced to $C$. Let us now assume that (say) $|g-e|=1$. If $|g-f|=2$, then it is easily seen that $O_{X,\G}$ reduces to $C$. Otherwise
 $|g-f|=1$, and there are two possibilities: if $q=2$, then a diagonal of the apartment $A$ contains $g$. Therefore the extension invariant consists of $C$ and of exactly one more edge corresponding to this diagonal---it can be shown that this geometry for the extension invariant characterizes $G_2$ among $A_2$ buildings with 3 missing chambers. The completion is then unique in this case as well (in fact, if there exist two distinct alternating 6-cycle $C$ and $C'$ in the invariant graph of $X,\G$, then their set of edges of type 0  must coincide, while their set of edges of type 1 must be disjoint). In other cases,  $q\geq 3$, and $O_{X,\G}$ is reduced to $C$.
\end{proof}

More generally, we have: 

\begin{theorem}\label{cor:unicity2}
Let $(X,\G)$ be a building of type $\tilde A_2$ with chambers missing and $(X,\G)\leadsto (X',\G')$ be an extension. Assume that the distance between equivariant missing chambers in $X'$ is at least 2. Then $(X,\G)\leadsto (X',\G')$ is the unique extension of $(X,\G)$ into a Euclidean building of type $\tilde A_2$.
\end{theorem}

\begin{proof}
If the distance between the missing chambers is at least 2, then the invariant graph $O_{X,\G}$ has a decomposition into connected components associated to each missing chamber, so we can apply inductively Theorem \ref{cor:unicity} to each of them after, choosing an ordering of the connected component of $O_{X,\G}$.  This produces a sequence of extensions
\[
(X,\G)\leadsto (X_2,\G_2)\leadsto\cdots (X_{n-1},\G_{n-1})\leadsto (X',\G')
\]
where $n$ is the number of missing chambers, and at each step the extension $(X_i,\G_i)\leadsto (X_{i+1},\G_{i+1})$ is unique into a building with chambers missing. Observe in particular that a different choice for the ordering of the connected components of $O_{X,\G}$ leads to the same building $(X',\G')$.
\end{proof}

This result will be important for our treatment of random buildings with chambers missing  in \cite{random},  both in the model with few missing chambers  and in the density model when the density parameter $d$ satisfies $d<1/2$.

\begin{remark} The above results can be seen as an ``extension rigidity" property for buildings with chambers missing. This appears to be a \emph{remanent} rigidity property.      
It seems remarkable indeed that a similar statement does not hold for certain groups \emph{of rank \sq} which act \emph{transitively} on the vertex set of their associated complexes. For an example, consider the two complexes $V_0^2$ and $\check V_0^2$ from Section 4 of \cite{rd}. Recall that they have the following presentation:
\begin{align*}
V^2_0=[[1, 2, 3], [1, 4, 5], [1, 6, 7], [2, 4, 6], [2, 8, 5], [3, 6, 8], [3, 7, 5], [4, 8, 7]]\\
\check V_0^2=[[1, 2, 3], [1, 4, 5], [1, 6, 7], [2, 6, 4], [2, 8, 5], [3, 6, 8], [3, 7, 5], [4, 8, 7]].
\end{align*}
Both are extensions, in the sense considered above, of the following complex of rank \sq\ with one chamber missing:
\begin{align*}
[[1, 2, 3], [1, 4, 5], [1, 6, 7],  [2, 8, 5], [3, 6, 8], [3, 7, 5], [4, 8, 7]],
\end{align*}
yet $V_0^2$ is not isomorphic to $\check V_0^2$.
\end{remark}

\section{The graph $G_5$ (the inverse pyramid)}\label{Sectinvert}

For $G_i$, $i=1\ldots 6$, $i\neq 5$, we have shown how to construct a CAT(0) simplicial complex whose automorphism group is transitive on vertices and whose link is isometric to $G_i$.  
In this section we  prove the following result. Part b) and part c) of the theorem are stated after the proof of part  a).

\begin{theorem}\label{th101} a) There is no  simplicial complex whose link at every vertex is simplicially isomorphic to the `inverse pyramid' $G_5$.
\end{theorem}

The mere assumption here is that every face is a triangle, with no additional further (e.g. metric) structure. This is in  sharp contrast with what happens for the other $G_i$ (in particular with $G_6$), and of course with the well-known fact that such complexes always exist as soon as  \emph{the valency of the link  is constant} (see \cite[Prop. 4.1]{BS} and references);  furthermore, the usual constructions provide CAT(0) complexes (typically infinitely many  isomorphism classes), while \ref{th101} exhibits a purely combinatorial obstruction.

\begin{proof}
Assume that $X$ is such a complex, let $p_0$ be a vertex. Let $p_1$ the vertex of $X$ corresponding to the apex of one of the  pyramids in the link at $p_0$, and denote $e=[p_0,p_1]$. The following holds:
\begin{itemize}
\item there are 3 triangles $t_1,t_2,t_3$ in $X$ with base $e$;
\item the edge of $t_k$ containing $p_0$ and distinct from $[p_0,p_1]$ is included in a unique triangle $t_k'$ of $X$ distinct from $t_k$ ($k=0,1,2$).
\end{itemize}
Let $u$ be the vertex of the link at $p_1$ corresponding to $e$. The structure of $G_5$ shows that there exists at least one neighbor $v$ of $u$ which is a vertex of valency 2. Let $k=0,1,2$ be the index for which $v$ belong to $t_k$, and let $p_2$ be the vertex of $t_k$ which does not belong to $e$. Then one sees that the link at $p_2$ cannot be simplicially isomorphic to $G_5$. 
\end{proof}

As the proof indicates, the obstruction may vanish if the construction rules are slightly modified. For example, one can allow other types of local geometry instead of confining oneself to $G_5$ (which could be used at $p_1$, in the notation of the proof), or allow more general faces than triangles (which could prevent  $p_2$ to exist). 
We will now see that these two variations of the original riddle indeed allow the construction of  complexes, and furthermore, that these complexes can be taken to be CAT(0).
As for the examples constructed in the previous sections, the main difficulty is to understand what are the explicit shapes and labelings that fit the requirements. 

\bigskip

\emph{b)  There is a simplicial complex  with links isomorphic to either $G_1$ or $G_5$ (where both link appear), and whose automorphism group has exactly two orbits of vertices.}

\bigskip

The description of this complex is more involved than the others. 
\begin{figure}[h]
\centering
\begin{tikzpicture}

\begin{scope}[xshift=-2cm,yshift=0cm]
\draw[-latex] (1,0) -- node[anchor=west]{\tiny{1}} (0.5,0.866);
\draw[-latex] (0.5,0.866) -- node[anchor=east]{\tiny{1}} (0,0);
\draw[-latex] (0,0) -- node[anchor=north]{\tiny{2}} (1,0);
\shade[ball color=black] (0,0) circle (0.3ex);
\shade[ball color=black] (1,0) circle (0.3ex);
\shade[ball color=black] (0.5,0.866) circle (0.3ex);
\end{scope}

\begin{scope}[xshift=-2cm,yshift=-1.5cm]
\draw[-latex] (1,0) -- node[anchor=west]{\tiny{9}} (0.5,0.866);
\draw[-latex] (0.5,0.866) -- node[anchor=east]{\tiny{9}} (0,0);
\draw[-latex] (0.5,-0.866) -- node[anchor=east]{\tiny{5}} (0,0);
\draw[-latex] (1,0) -- node[anchor=west]{\tiny{3}} (0.5,-0.866);
\shade[ball color=black] (0,0) circle (0.3ex);
\shade[ball color=black] (1,0) circle (0.3ex);
\shade[ball color=black] (0.5,0.866) circle (0.3ex);
\shade[ball color=black] (0.5,-0.866) circle (0.3ex);
\end{scope}

\begin{scope}[xshift=0cm,yshift=.5cm]
\draw[-latex] (1,0) -- node[anchor=west]{\tiny{9}} (0.5,0.866);
\draw[-latex] (0.5,0.866) -- node[anchor=east]{\tiny{7}} (0,0);
\draw[-latex] (0,0) -- node[anchor=north]{\tiny{4}} (1,0);
\shade[ball color=black] (0,0) circle (0.3ex);
\shade[ball color=black] (1,0) circle (0.3ex);
\shade[ball color=black] (0.5,0.866) circle (0.3ex);
\end{scope}
\begin{scope}[xshift=0cm,yshift=-1cm]
\draw[-latex] (1,0) -- node[anchor=west]{\tiny{6}} (0.5,0.866);
\draw[-latex] (0.5,0.866) -- node[anchor=east]{\tiny{6}} (0,0);
\draw[-latex] (0,0) -- node[anchor=north]{\tiny{2}} (1,0);
\shade[ball color=black] (0,0) circle (0.3ex);
\shade[ball color=black] (1,0) circle (0.3ex);
\shade[ball color=black] (0.5,0.866) circle (0.3ex);
\end{scope}
\begin{scope}[xshift=0cm,yshift=-2.5cm]
\draw[-latex] (1,0) -- node[anchor=west]{\tiny{8}} (0.5,0.866);
\draw[-latex] (0.5,0.866) -- node[anchor=east]{\tiny{8}} (0,0);
\draw[-latex] (0,0) -- node[anchor=north]{\tiny{2}} (1,0);
\shade[ball color=black] (0,0) circle (0.3ex);
\shade[ball color=black] (1,0) circle (0.3ex);
\shade[ball color=black] (0.5,0.866) circle (0.3ex);
\end{scope}

\begin{scope}[xshift=3cm]
\path (0:1cm) coordinate (P0);
\path (1*60:1cm) coordinate (P1);
\path (2*60:1cm) coordinate (P2);
\path (3*60:1cm) coordinate (P3);
\path (0,0) coordinate (P4);
\draw [-latex] (P1) -- node[anchor=west]{\tiny{7}}  (P0);
\draw [-latex] (P1) -- node[anchor=south]{\tiny{3}}  (P2);
\draw [-latex] (P3) -- node[anchor=east]{\tiny{5}}  (P2);
\draw [-latex] (P4) -- node[anchor=north]{\tiny{4}}  (P3);
\draw [-latex] (P0) -- node[anchor=north]{\tiny{8}}  (P4);
\foreach \i in {0,...,4} \shade[ball color=black] (P\i) circle (0.3ex);
\end{scope}

\begin{scope}[xshift=3cm,yshift=-2cm]
\path (0:1cm) coordinate (P0);
\path (1*60:1cm) coordinate (P1);
\path (2*60:1cm) coordinate (P2);
\path (3*60:1cm) coordinate (P3);
\path (0,0) coordinate (P4);
\draw [-latex] (P1) -- node[anchor=west]{\tiny{3}}  (P0);
\draw [-latex] (P1) -- node[anchor=south]{\tiny{5}}  (P2);
\draw [-latex] (P3) -- node[anchor=east]{\tiny{4}}  (P2);
\draw [-latex] (P4) -- node[anchor=north]{\tiny{6}}  (P3);
\draw [-latex] (P0) -- node[anchor=north]{\tiny{7}}  (P4);
\foreach \i in {0,...,4} \shade[ball color=black] (P\i) circle (0.3ex);
\end{scope}
\end{tikzpicture}
\caption{}\label{fig5}
\end{figure}
It is constructed out of three shapes: triangles, lozenges and trapezes, as shown on Fig.\ \ref{fig5}. When simplicializing this complex into equilateral triangles, we obtain a simplicial complex $X=\tilde V$ satisfying b) above. Furthermore, $(X,\G)$ is not a building with chambers missing (cf. Section \ref{locglob}), it doesn't have the isolated flats property. We leave to the reader to check these claims and observe the following interesting fact: 

\bigskip

 \emph{b') The fundamental group of the complex described on Figure \ref{fig5} is of finite index in the full automorphism group of its universal cover.}

\begin{proof}
Let $\theta$ be a automorphism of $X$ which fixes the sphere $S_1$ of radius 1 around a vertex $x\in X$. It follows from Lemma \ref{L:emma - aut-gi-fixes tripod} that $\theta$ fixes every link of every vertex of $S_1$ which has order at least 2.  In particular, $\theta$ fixes the sphere of radius 2 around $x$. A straightforward induction using Lemma \ref{L:emma - aut-gi-fixes tripod} shows that $\theta$ fixes the complex $X$ pointwise.
This shows that the group $\G$ is of finite index in the automorphism group of $X$.\end{proof}

\begin{remark}
Since Lemma \ref{L:emma - aut-gi-fixes tripod} applies to any of the graphs $G_1,\ldots, G_6$, the proof of Assertion b') shows that, if $\G$ be a group on a complex $X$ with $X/\G$ compact  and   if the link of every vertex in $X$ is isomorphic to one of the 6 graphs $G_1,\ldots, G_6$, then $\G$ is of finite index in the full automorphism group of $X$.  In fact, it is enough to have a complex $X$ whose links satisfy the conclusion of Lemma \ref{L:emma - aut-gi-fixes tripod} for the proof of Assertion b') to imply that $\G$ is of finite index in the automorphism group of $X$. We observe that the conclusion of Lemma \ref{L:emma - aut-gi-fixes tripod} fails for many spherical buildings with chambers missing. 
\end{remark}

Let us now turn to the second modified riddle, whose solution is also used in the next section. 
Our construction shows that the analog of Part a) fails if we allow faces to be triangles and hexagons.

\bigskip

\emph{c$_0$)  There is a CAT(0) polyhedral complex whose link at every vertex is isometric to $G_5$, and whose automorphism group is transitive on vertices. Faces of these complex are all isometric to equilateral triangles or regular hexagons of the Euclidean plane.}

\bigskip

A solution is given by the complex $V$ with presentation described on Fig.\ \ref{fig4}, 
\begin{figure}[h]
\centering
\begin{tikzpicture}
\begin{scope}[xshift=2.5cm]
\path (0:1cm) coordinate (P0);
\path (1*60:1cm) coordinate (P1);
\path (2*60:1cm) coordinate (P2);
\path (3*60:1cm) coordinate (P3);
\path (4*60:1cm) coordinate (P4);
\path (5*60:1cm) coordinate (P5);

\draw [-latex] (P0) -- node[anchor=west]{\tiny{1}}  (P1);
\draw [-latex] (P2) -- node[anchor=south]{\tiny{4}}  (P1);
\draw [-latex] (P2) -- node[anchor=east]{\tiny{2}}  (P3);
\draw [-latex] (P4) -- node[anchor=east]{\tiny{4}}  (P3);
\draw [-latex] (P4) -- node[anchor=north]{\tiny{3}}  (P5);
\draw [-latex] (P0) -- node[anchor=west]{\tiny{4}}  (P5);
\foreach \i in {0,...,5} \shade[ball color=black] (P\i) circle (0.3ex);
\end{scope}

\begin{scope}[xshift=-2.5cm,yshift=-.5cm]
\draw[-latex] (1,0) -- node[anchor=west]{\tiny{2}} (0.5,0.866);
\draw[-latex] (0.5,0.866) -- node[anchor=east]{\tiny{3}} (0,0);
\draw[-latex] (0,0) -- node[anchor=north]{\tiny{1}} (1,0);
\shade[ball color=black] (0,0) circle (0.3ex);
\shade[ball color=black] (1,0) circle (0.3ex);
\shade[ball color=black] (0.5,0.866) circle (0.3ex);
\end{scope}

\begin{scope}[xshift=-.5cm,yshift=-.5cm]
\draw[-latex] (1,0) -- node[anchor=west]{\tiny{3}} (0.5,0.866);
\draw[-latex] (0.5,0.866) -- node[anchor=east]{\tiny{2}} (0,0);
\draw[-latex] (0,0) -- node[anchor=north]{\tiny{1}} (1,0);
\shade[ball color=black] (0,0) circle (0.3ex);
\shade[ball color=black] (1,0) circle (0.3ex);
\shade[ball color=black] (0.5,0.866) circle (0.3ex);
\end{scope}
\end{tikzpicture}
\caption{}\label{fig4}
\end{figure}
that is, $V$ is built out of two equilateral triangles and an hexagon, which are assembled according to labels, respecting orientation.

Of course, this second complex is also a solution to the first modified riddle, where the second link is a circle of length $2\pi$.

\medskip

We need a definition before continuing.

\begin{definition}\label{weakbmiss}
We call \emph{weak building with chambers missing} a couple $(X,\G)$ satisfying all conditions of Definition \ref{bmiss}, except perhaps for \ref{bmiss}.b) which we weaken to
\begin{itemize}
\item[\ref{bmiss}.\= b)] $\lambda(X)\cup  (\G' C_1)\cup\ldots \cup (\G'C_n)= X'$.
\end{itemize}
\end{definition}

This allows, for example, to remove chambers together with some of their panels (to the extent to which the others conditions still hold). 

\medskip

Here is the last part of Theorem \ref{th101}:

\medskip

\emph{c$_1$) Let $\G$ be the fundamental group $\G=\pi_1(V)$ of the complex described on Fig.\ \ref{fig4} and $X=\tilde V$ be its universal cover. Then:
\begin{enumerate}
\item $(X,\G)$ is not a building with chambers missing;
\item $(X,\G)$ is a weak building with chambers missing of type $\tilde A_2$;
\item $X$ has the isolated flat property. 
\end{enumerate}}

\begin{proof}
(1) is straightforward. To show (2), we observe that  the complex  obtained from the exotic building constructed \cite[Section 3]{toulouse} by removing all faces which have more than two  white vertices (see Fig.\ 23 in \cite{toulouse}) is isometric to $V$. We recall this exotic buildings has two types of vertices, black or white, depending on whether the isomorphism type of the 2-sphere is that of the building of $\PGL_3(\QI_2)$ or of $\PGL_3(\FI_2((y)))$, and that its isometry group is transitive on vertices of a given type. 

 Finally, that $X$ has isolated flats readily follows from Lemma \ref{lem:isfl}: the inverse pyramid $L$ satisfies   $\rk_\infty(L)\leq 1^+$ (compare Prop. \ref{rank3}). 
\end{proof}

\begin{remark}\label{rem:b4} 1) It would be interesting to have a local criterion (similar to \ref{th:loccrit}) for weak buildings with chambers missing. 

2) Let $\G=B_4/Z$ be the quotient of the 4-string braid group by its center. It is known (see \cite[Fig.\ 5]{b4} and references therein) that $\G$ acts on a CAT(0) complex $X$ of dimension 2  whose link is a building of type $A_2$ with 5 chambers missing. The action of $\G$ is transitive on vertices and proper (but not free).  With a generalisation of Definition \ref{bmiss} to proper actions, one can show that  that $(X,\G)$ is not a weak building with missing chamber of type $\tilde A_2$. This should be compared to Prop. 20 and Rem. 21 in \cite{b4}, which can be generalized to the extension map $(X,\G)\leadsto (X',\G')$.
\end{remark}

\section{The Haagerup property for buildings with chambers missing}\label{Sect:properties}

Our main result of this section is the following theorem, which gives a new example of a group with the Haagerup property. For information on the latter, we refer to \cite{Valette-book}.

\begin{theorem}\label{th-haagerup} Let $\G$ be the fundamental group $\G=\pi_1(V)$ of the complex described on Fig.\ \ref{fig4}. Then $\G$ has the Haagerup property.
\end{theorem}

We note that since the corresponding simplicial complex $X=\tilde V$ contains flats, the building $(X,\G)$ is non degenerate (as a weak building with chambers missing).

\begin{proof}
We construct a $\G$ invariant family of geometric walls in $X$ with the following properties:
\begin{enumerate}
\item[(A)] the set of walls separating any two points is finite;
\item[(B)] the number of walls separating any two points is goes to infinity with the distance between these points.
\end{enumerate}
This is known, by results of Haglund and Paulin, to imply the Haagerup property for any group acting geometrically on $X$  (see \cite[p. 5 and Prop. 7.4.2]{Valette-book}). This criterion applies for example to CAT(0) cube complexes, and more generally to even polyhedral complexes (the complex $X$ we consider here is not even). 

Our walls in $X$ are also defined locally and will leave footprints on faces as described on the following figure:

\begin{figure}[h]
\centering
\begin{tikzpicture}
\begin{scope}[xshift=2.5cm]
\path (0:1cm) coordinate (P0);
\path (1*60:1cm) coordinate (P1);
\path (2*60:1cm) coordinate (P2);
\path (3*60:1cm) coordinate (P3);
\path (4*60:1cm) coordinate (P4);
\path (5*60:1cm) coordinate (P5);

\draw  (P0) --   (P1);
\draw  (P2) --   (P1);
\draw (P2) --   (P3);
\draw  (P4) --  (P3);
\draw  (P4) --  (P5);
\draw  (P0) --   (P5);
\draw[dotted] (P2) --  (P4);
\foreach \i in {0,...,5} \shade[ball color=black] (P\i) circle (0.3ex);
\end{scope}

\begin{scope}[xshift=6cm]
\path (0:1cm) coordinate (P0);
\path (1*60:1cm) coordinate (P1);
\path (2*60:1cm) coordinate (P2);
\path (3*60:1cm) coordinate (P3);
\path (4*60:1cm) coordinate (P4);
\path (5*60:1cm) coordinate (P5);

\draw  (P0) --   (P1);
\draw  (P2) --   (P1);
\draw (P2) --   (P3);
\draw  (P4) --  (P3);
\draw  (P4) --  (P5);
\draw  (P0) --   (P5);
\draw[dotted] (0,.866) --  (0,-.866);
\foreach \i in {0,...,5} \shade[ball color=black] (P\i) circle (0.3ex);
\end{scope}

\begin{scope}[xshift=-1cm,yshift=-.5cm]
\draw (1,0) --  (0.5,0.866);
\draw (0.5,0.866) --  (0,0);
\draw (0,0) --  (1,0);
\draw[dotted] (0.5,0.866) --  (.5,0);
\shade[ball color=black] (0,0) circle (0.3ex);
\shade[ball color=black] (1,0) circle (0.3ex);
\shade[ball color=black] (0.5,0.866) circle (0.3ex);
\end{scope}
\end{tikzpicture}
\caption{Walls in $X$}\label{fig:wall}
\end{figure}
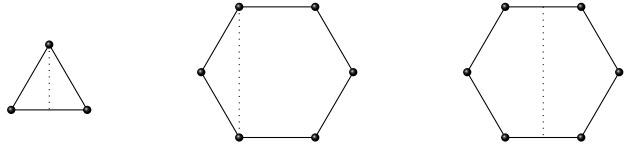

Namely, for each triangle $t\in X$ we define three geodesic trees $U_t, V_t, W_t$ associated to each of the three medians of $t$. We describe the construction for $U_t$, the two others being symmetric. Let $m$ be a fixed median of $t$.

Construct by recurrence a subtree  $U_t^n$ of $X$ with leaves $E_t^n\subset U_t^n$ such that $U_t^{n+1}=U_t^n\cup E_t^{n+1}$ (vertex set equality) in the following way.
Start with $U_t^{1}=m\subset t$  and $E_t^{1}=\del m\subset \del t$, and assume we have constructed $U_t^{n},E_t^{n}$ for some $n\geq 1$. 
For each $e\in E_n$ denote by $L_e$ the link of $e$ in $X$. Thus $L_e$ is either a union of edges of length $\pi$ glued on their boundaries, or the inverse pyramid (with all edges of length $\pi/3$). Let $r_e\in L_e$ be the point corresponding to the unique edge of $U_t^n$ containing $e$. We define $E_{n+1}'(e)$ to be the set of vertices of $L_e$ at distance exactly $\pi$ from $r_e$ on a circle of length $2\pi$ containing $r_e$. 
 
\begin{lemma}\label{lem105} The point $r_e$ and the points in $E_{n+1}'(e)$ are middle points of edges of $L_e$ and their pairwise distance is at least $\pi$. Furthermore,
the graph $L_e\backslash \{E_{n+1}'(e)\cup r_e\}$ is a union of two disjoint connected \emph{trees}.
\end{lemma}

This can be proved by recurrence and a direct inspection of the inverse pyramid. We note that in particular, the map associating to $r_e$ the subset $E_{n+1}'(e)$ of $L_e$ is local symmetry in the sense of \cite{cras}, i.e.\  it satisfies  (1) and (2) on p. 280. Another example where this property can be exploited to prove the Haagerup property is described in \cite{cras}.

For $e'\in E_{n+1}'(e)$ let $f_{e'}$ be the face of $X$ containing $e'$ and $p_{e'}\in \del f_{e'}$ be the point opposite to $e$ in the direction $e'$. We let 
\[
E_{n+1}=\bigcup_{e\in E_n}\bigcup_{e'\in E_n'(e)} p_{e'}
\]
and define $U_t^{n+1}$ as the geodesic closure of $E_{n+1}$. It is easily checked that $U_t^{n+1}$ is a geodesic tree with boundary $E_{n+1}$ and such that  $U_t^{n+1}=U_t^n\cup E_t^{n+1}$. We let 
\[
U_t=\bigcup_n U_n.
\]
By construction, $U_t$ is a geodesic tree. Its footprint on faces is either empty or as represented on Fig.\ \ref{fig:wall}.

\begin{lemma}\label{HaagerupTriangle}
Let $\Delta$ be a totally geodesic triangle of $X$. Assume that the intersection of $U_t$ and  $\Delta$ contains contains a nonempty open segment of $U_t$. Then the two end points of the segment $\Delta\cap U_t$ are contained in the boundary of $\Delta$.   
\end{lemma}

\begin{proof}
Let $I=\Delta\cap U_t$, so that by assumption $I$ is a closed geodesic segment with nonempty interior, say $I=[p,q]$. We must show that both $p$ and $q$ belong to the boundary of $\Delta$. So let us assume for instance that $p$ is interior to $\Delta$ (the other case being similar).

By construction of $U_t$ we have either that $p$ belongs  to an open edge of $X$, or that $p$ is a vertex of $X$. In the first case, the link at $p$ is a reunion of edges of length $\pi$. Therefore, if $p$ is interior to $\Delta$, then $\Delta$ is flat at $p$ and corresponds to a circle of length $2\pi$ in $L_p$. The point on this circle at distance $\pi$ from the point corresponding to $I$ in $L_p$ projects to a point $p'$ in $\Delta$, $p'\neq p$. But then by construction, $p'\in U_t$ and this contradicts the fact that $I=\Delta\cap U_t$. 

Assume now that $p$ is a vertex, so that $L_p$ is the inverse pyramid and $\Delta$ now corresponds to a circle of length $\geq 2\pi$ in $L_p$. By Lemma \ref{lem105}, we can choose a point $e'$ at distance $\pi$ from the point corresponding to $I$ on this circle, and a corresponding point $p'$  of $\Delta$ which corresponds to $e'$, $p'\neq p$, and such that $p'\in U_t$. This gives a contradiction and proves the lemma. 
\end{proof}

\begin{lemma}
The set $X\backslash U_t$ has exactly two connected components. 
\end{lemma}

\begin{proof}
Assume first that $X\backslash U_t$ has at least 3 connected components and let us find a contradiction. Take three points in three distinct components and consider the geodesic triangle $\Delta$ between these three points. If the intersection of  $\Delta$ and $U_t$ contains no non trivial subsegment of $U_t$, then we can easily find a path in $\Delta$ between its vertices which doesn't intersect $U_t$, and this gives a contradiction.  Otherwise we are in position to apply Lemma \ref{HaagerupTriangle}. Let $[p,q]$ be the intersection of $\Delta$ and $U_t$, where $p,q\in \del \Delta$. If $[p,q]$ intersects only two sides of $\Delta$ or doesn't intersect the interior of $\Delta$, then we readily get a contradiction. If not, it follows that there exists a point $r\in ]p,q[$ such that the points $p,q,r$ belongs to the three different sides of $\Delta$. Let $[A,B]$ be the side of $\Delta$ containing $r$. We will find a path in $X$ from $A$ to $B$ which doesn't intersect $U_t$. We may assume that $]p,r[$ is included in the interior of $\Delta$ and consider, symmetrically, the point $r'\in [q,r]$ of $\del \Delta$ such that $]q,r'[$ is included in the interior of $\Delta$ (possibly, $r=r'$). It is easy to see that the path $\beta_r$ (resp. $\beta_{r'}$) of the link of $r$ (resp.  $r'$)  which correspond to the disk $\Delta$ has length $>\pi$. Furthermore, it is included in a  circle $\gamma_r$ (resp. $\gamma_{r'}$) of length $>2\pi$ such that $\gamma_r\backslash \beta_r$ (resp. $\gamma_{r'}\backslash \beta_{r'}$) does not contain any point corresponding to $U_t$. Therefore, it is possible to extend both $\beta_r$ and $\beta_{r'}$  within $\gamma_r$  and $\gamma_{r'}$ on both sides, and find two paths in $X$ in neighborhoods of $r$ and $r'$ corresponding to the new endpoints  of $\beta_r$  and $\beta_{r'}$, in such a way that these paths do not intersect $U_t$. Furthermore, we can iterate this construction around each vertex of $]r,r'[$ if necessary. This paths can then be extended to construct the desired path from $A$ to $B$. So again, we obtain a contradiction. Thus $X\backslash U_t$ has at most two connected components.

Assume now that $X\backslash U_t$ is connected, and let $p,q$ be the two vertices of $t$ which are not in $U_t$. By assumption, there is a path $\gamma$ from $p$ to $q$ which does not intersect $U_t$. We may assume that $\gamma$ is piecewise linear. Since $X$ is contractible, there is a piecewise linear homotopy $(H(a,\cdot))_{a\in [0,1]}$ between $H(0,\cdot)=\gamma$ and $H(1,\cdot)=$ the geodesic segment from $p$ to $q$. Let $a_0\in [0,1]$ be the smallest value such that the path $H(a,\cdot)$ intersect $U_t$ for all $a\geq a_0$, and denote by $H(a_0,b_0)$ an intersection point with $U_t$ at time $a_0$. The local geometry around $t$ shows that $a_0<1$. Furthermore,  $H(a_0,b_0)$ is a vertex of $X$ whose link $L$ is the inverse pyramid.  Suppose that there exists $\e>0$ sufficiently small, such that $H(a_0-\e,\cdot)$ does not intersect $U_t$. We may assume that $H(a_0-\e',\cdot)$ does not intersect $U_t$ for all $0<\e'<\e$. Then the projection of the subset $\{H(a,b)\mid a_0-\e\leq a\leq a_0+\e, b\in [0,1]\}$ of $X$  into the link of $H(a_0,b_0)$ contains a circle of length $\geq 2\pi$ in  $L$ which contains only one point issued from the tree $U_t$. But this  contradicts Lemma \ref{lem105}. Hence, the points $p$ and $q$ are not in the same connected component.

This shows that $X\backslash U_t$ has exactly two connected components.    
\end{proof}

Associated to $U_t$ are two walls in $X$: denote by $C_1$ and $C_2$ the two connected components of $X\backslash U_t$ given by the above lemma, then the first wall is $(C_1,U_t\cup C_2)$ and the second is $(C_1\cup U_t, C_2)$. 

Similarly, we define two walls for each $V_t$, $W_t$, $t\in X$.

Now (A) and (B) follow from the fact that every maximal geodesic segment of a face of $X$ intersect transversally a finite, but nonempty, family of trees $(U_t,V_t,W_t)$ where $t$ runs over triangles of $X$. 
\end{proof}

\begin{remark} 
  Other properties of buildings with chambers missing are non generic but hold in many concrete cases, for example: 
    deficiency 1  (e.g. one-relator groups with 2 generators, or 3 generated groups with two relators), indicability, positive first $\ell^2$ Betti number, etc.  
We also remark that, in Gromov's density model, the Haagerup property holds with overwhelming probability in density $d<1/6$, for these groups act freely on a CAT(0) cube complex with compact quotient \cite{OW}. 
\end{remark}

\begin{question}
Let $(X,\G)$ be a  building with (at least one) chambers missing. When does $\G$ have the Haagerup property? When does $\G$ act (say properly with compact quotient) on a cube complex?   
\end{question}

It seems difficult to formulate a general criterion. Related issues are addressed in  the recent paper \cite{HW} of Hruska and Wise, along with an in-depth study of cubulating techniques for groups acting on spaces with walls.


\begin{thebibliography}{00}


\bibitem{BS} Ballmann, W.; Swiatkowski, J. On $L\sp 2$-cohomology and property (T) for automorphism groups of polyhedral cell complexes. Geom. Funct. Anal. 7 (1997), no. 4, 615--645.



\bibitem{toulouse} Barr\'e, S. Immeubles de Tits triangulaires exotiques. (French) [Exotic triangular Tits buildings] Ann. Fac. Sci. Toulouse Math. (6) 9 (2000), no. 4, 575--603.

\bibitem {cras} Barr\'e, S. La propri\'et\'e de Haagerup pour des complexes localement sym\'etriques. C. R. Acad. Sci. Paris S\'er. I Math. 333 (2001), no. 4, 279--284.

\bibitem{rd} Barr\'e S., Pichot M., Intermediate rank and property RD, http://arxiv.org/abs/0710.1514.

\bibitem{b4} Barr\'e S., Pichot M., The 4-string braid group has property RD and exponential mesoscopic rank,  http://arxiv.org/abs/0809.0645.

\bibitem{random} Barr\'e S., Pichot M., Random groups and nonarchimedean lattices, in preparation.


\bibitem{HV}   Bekka, Bachir; de la Harpe, Pierre; Valette, Alain Kazhdan's property (T). New Mathematical Monographs, 11. Cambridge University Press, Cambridge, 2008.


\bibitem{BH} Bridson M., Haefliger A.,
Metric spaces of non-positive curvature. 
Grundlehren der Mathematischen Wissenschaften [Fundamental Principles of Mathematical Sciences], 319. Springer-Verlag, Berlin, 1999.

\bibitem{brown} Brown, Kenneth S. Buildings. Reprint of the 1989 original. Springer Monographs in Mathematics. Springer-Verlag, New York, 1998. 

\bibitem{CS} Cartwright, Donald I.; Steger, Tim A family of $\tilde A\sb n$-groups. Israel J. Math. 103 (1998), 125--140.

\bibitem{Valette-book} Cherix, Pierre-Alain; Cowling, Michael; Jolissaint, Paul; Julg, Pierre; Valette, Alain Groups with the Haagerup property. Gromov's a-T-menability. Progress in Mathematics, 197. Birkh\"auser Verlag, Basel, 2001. 



\bibitem{FH} Feit, Walter; Higman, Graham
The nonexistence of certain generalized polygons. 
J. Algebra 1 1964 114--131. 


\bibitem{gromov0} Gromov M., Hyperbolic groups. Essays in group theory, 75--263, 
Math. Sci. Res. Inst. Publ., 8, Springer, New York, 1987.

\bibitem{gromov1} Gromov, M. Geometric group theory. Vol. 2. 
Proceedings of the symposium held at Sussex University, Sussex, July 1991. Edited by Graham A. Niblo and Martin A. Roller. London Mathematical Society Lecture Note Series, 182. Cambridge University Press, Cambridge, 1993.
 
 
 \bibitem{HP} F. Haglund and F. Paulin, Simplicit\'e de groupes d'automorphismes d'espaces \`a courbure n\'egative. In The Epstein birthday schrift, pages 181--248. Geom. Topol., Coventry, 1998.
 
\bibitem{HK} Hruska, G. Christopher; Kleiner, Bruce
Hadamard spaces with isolated flats.  
Geom. Topol. 9 (2005), 1501--1538



\bibitem{HW} Hruska, C.; Wise, D. 
Finiteness properties of cubulated groups,
preprint.

 \bibitem{mostow} Mostow, G. D. Strong rigidity of locally symmetric spaces. Annals of Mathematics Studies, No. 78. Princeton University Press, Princeton, N.J.; University of Tokyo Press, Tokyo, 1973.

\bibitem{OW}  Ollivier, Yann, Wise, Daniel T. Cubulating random groups at density less than 1/6, Trans. Amer. Math. Soc., to appear. 


\bibitem{Ronan}  Ronan, Mark Lectures on buildings. Perspectives in Mathematics, 7. Academic Press, Inc., Boston, MA, 1989.

\bibitem{serre} Serre, Jean-Pierre Arbres, amalgames, ${\rm SL}\sb{2}$. R\'edig\'e avec la collaboration de Hyman Bass. Ast\'erisque, No. 46. Soci\'et\'e Math\'ematique de France, Paris, 1977.

\bibitem{tits-sphere} Tits, Jacques Spheres of radius $2$ in triangle buildings. I. Finite geometries, buildings, and related topics (Pingree Park, CO, 1988), 17--28, Oxford Sci. Publ., Oxford Univ. Press, New York, 1990.


\end{thebibliography}
\end{document}